\documentclass[a4paper,12pt]{article}
\usepackage[utf8]{inputenc}
\usepackage[english]{babel}
\usepackage{enumitem}
\usepackage{color}
\usepackage{amsmath}
\usepackage{amssymb}
\usepackage{amsthm}
\usepackage{amsfonts}
\usepackage{epsfig}
\usepackage{subcaption}
\usepackage{caption}
\usepackage{graphicx}
\usepackage{epstopdf}
\usepackage{multirow}
\usepackage{upgreek}
\usepackage{placeins}
\usepackage{mathrsfs}
\usepackage{float}%
\usepackage{booktabs}
\usepackage{multirow}
\usepackage{adjustbox}
\usepackage{subcaption}
\usepackage{pifont}

\usepackage{caption}
\newtheorem{dfn}{Definition}[section]
\newtheorem{lem}{Lemma}[section]

\newtheorem{alg}{Algorithm}[section]
\newtheorem{theorem}{Theorem}[section]
\newtheorem{example}{Example}[section]

\usepackage{makecell}
\theoremstyle{definition}
\newtheorem{asm}{Assumption}[section]

\newtheorem{rem}{Remark}[section]

\newtheorem{fact}{Fact}[section]
\usepackage{tabularx}  
\usepackage{booktabs}

\DeclareMathOperator{\diag}{diag}
\DeclareMathOperator*{\zer}{zer}
\DeclareMathOperator*{\Fix}{Fix}
\DeclareMathOperator*{\gra}{gra}

\DeclareMathOperator*{\Id}{Id}
\DeclareMathOperator*{\range}{range}

\DeclareMathOperator*{\dom}{dom}

\DeclareMathOperator*{\slt}{slt}
\allowdisplaybreaks
\begin{document}
	\title{\textbf{A reflected forward-backward splitting algorithmic framework{\thanks{Supported by Scientific Research Project of Tianjin Municipal Education Commission (2020ZD02)}}}}
	\author{ \sc \normalsize Haowen Zheng$^a${\thanks{ email: zhw20251122@163.com}},\,\,
		Yongyu Fu$^b${\thanks{ email: fyy000611@163.com}},\,\,
		Qiao-Li Dong$^b${\thanks{Corresponding author. email: dongql@lsec.cc.ac.cn}}\,\,\,\,and\,\,
		Shuangbao Li$^c${\thanks{email: shuangbaoli@yeah.net}}\\
		\small 	$^a$ School of Light Industry and Engineering, South China University of Technology,\\ 	\small  Guangzhou 510640, China,\\
		\small $^b$College of Science and
		\small $^c$Research Institute of Science and Technology,\\
		\small Civil Aviation University of China, Tianjin 300300, China,
	}
	\date{}
	\maketitle
	
	\begin{abstract}
		{In this paper, we propose a reflected forward-backward splitting algorithmic framework for finding  a zero of the sum of finitely many monotone operators, including maximally monotone operators, cocoercive operators, and monotone and Lipschitz continuous operators.  We provide a unified convergence analysis under mild conditions, eliminating the need to analyze the convergence of each algorithm individually. The heuristic strategies for matrix selections are proposed through a numerical experiment, based on which a new algorithm is derived.
			A further numerical experiment on the regularized saddle-point problem is then presented to demonstrate the effectiveness of the proposed algorithm.
		}
	\end{abstract}
	
	
	\section{Introduction}
	In this paper, we focus on the structured inclusion problem in a real Hilbert space $\mathcal{H}$, which is to find $x\in \mathcal{H}$ such that
	\begin{equation}
		\label{ABC}
		0  \in  \sum_{i=1}^{n}A_i(x)+\sum_{i=1}^{m}B_i(x) +\sum_{i=1}^{l}C_i(x),
	\end{equation}
	where $A_i: \mathcal{H} \rightrightarrows \mathcal{H}$ are  maximally monotone, $i=1,\dots,n$,  $B_i: \mathcal{H}\to \mathcal{H} $ are cocoercive,  $i=1,\dots,m$ and $C_i: \mathcal{H}\to \mathcal{H} $ are monotone and Lipschitz continuous, $i=1,\dots,l$.
	The inclusion problem \eqref{ABC} provides a unified framework for modeling a wide class of structured optimization problems. In the following, we present two such examples.
	
	\begin{example}\rm \textbf{(Structured saddle-point problem)}
		\rm
		We consider a structured saddle-point problem:
		\begin{equation}\label{saddle}
			\min_{u\in\mathcal{H}_1} \max_{v\in\mathcal{H}_2}\sum_{i=1}^{n} (g_{i1}(u)-h_{i1}(v))+\sum_{i=1}^{l} \Psi _i(u,v)+\sum_{i=1}^{m}(g_{i2}(u)-h_{i2}(v)),
		\end{equation}
		where $g_{i1}, i=1,\dots,n :\mathcal{H}_1\to (-\infty ,+\infty ]$, $ h_{i1}, i=1,\dots,n :\mathcal{H}_2\to (-\infty ,+\infty ]$ are proper, lower semicontinuous (lsc)  and convex,  $g_{i2}, i=1,\dots,m :\mathcal{H}_2\to (-\infty ,+\infty )$, $h_{i2}, i=1,\dots,m :\mathcal{H}_1\to (-\infty ,+\infty )$ are convex and differentiable with Lipschitz continuous gradients, and $\Psi_i,i=1,\dots,l: \mathcal{H}_1\times\mathcal{H}_2\to (-\infty ,+\infty ]$ are differentiable convex-concave functions with Lipschitz continuous gradients. Assuming that a saddle-point exists, \eqref{saddle} can be posed as \eqref{ABC} in the space $\mathcal{H}=\mathcal{H}_1\times\mathcal{H}_2$ with
		$$A_i(u,v)=\begin{pmatrix}\partial g_{i1}(u)
			\\
			\partial h_{i1}(v)
		\end{pmatrix},\ B_i(u,v)=\begin{pmatrix}\nabla  g_{i2}(u)
			\\
			\nabla h_{i2}(v)
		\end{pmatrix}\ \hbox{and}\ C_i(u,v)=\begin{pmatrix}\nabla_u  \Psi_{i}(u,v)
			\\
			-\nabla_v  \Psi_{i}(u,v)
		\end{pmatrix}.
		$$
		Note that $A_i,i=1,\ldots,n$ are maximally monotone (\cite{Rockafellar1970}, \cite[Proposition 23.18]{BC2011}),   $B_i,i=1,\ldots,m$ are cocoercive  (see, e.g. \cite[Theorem 18.15]{BC2011}), and  $C_i,i=1,\ldots,l$ are monotone due to \cite[Theorem 2] {saddlemin}, and Lipschitz continuous, but generally not cocoercive.
		In fact, the problems \eqref{saddle} considered by many researchers can be viewed as special cases of \eqref{saddle} \cite{primal-2026}.
	\end{example}
	
	\begin{example}\rm \textbf{(Structured variational inequality problem)}
		\rm
		Consider a structured variational inequality problem:
		\begin{equation}\label{4}
			\begin{aligned}
				\text{find}~x^*\in\mathcal{H}~\text{such that}\sum_{i=1}^{n}f_i(x)&-\sum_{i=1}^{n}f_i(x^*)+\sum_{i=1}^{m}\langle B_i(x^*),x-x^*\rangle\\
				&+\sum_{i=1}^{l}\langle C_i(x^*),x-x^*\rangle\geq 0,\quad \forall x\in\mathcal{H},
			\end{aligned}
		\end{equation}
		where $f_1,...,f_n:\mathcal{H}\to(-\infty,+\infty]$ are proper, lsc and convex,  $B_i,i=1,\ldots,m:\mathcal{H}\to\mathcal{H}$ are cocoercive and $C_i,i=1,\ldots,l:\mathcal{H}\to\mathcal{H}$ are monotone and Lipschitz continuous. By standard arguments in convex analysis, the above variational inequality is equivalent to the inclusion problem (\ref{ABC}), where the operators $A_i=\partial f_i$
		are maximally monotone. Let $D_1, \dots, D_n \subseteq \mathcal{H}$ be nonempty, closed and convex sets and $f_i=\iota_{D_i}$ be the indicator of $D_i$, $i=1,\ldots,n$. Then the problem \eqref{4} reduces a typical representative example of the variational inequality problem, with its specific formulation stated below:
		\begin{equation*}\label{ques}
			\begin{aligned}
				\text{find}~x^*\in D~\text{such that}\sum_{i=1}^{m}\langle B_i(x^*),x-x^*\rangle+\sum_{i=1}^{l}\langle C_i(x^*),x-x^*\rangle\geq 0,~ \forall x\in D,
			\end{aligned}
		\end{equation*}
		where $D:=\bigcap_{i=1}^{n} D_i.$  In this way, the set $D$ can be handled through its simpler component sets
		$D_1,\ldots,D_n$.  The problems \eqref{ques} considered by many researchers can be viewed as special cases of \eqref{saddle} \cite{Alac2021}.
	\end{example}

	When $n=1$, $m=l=0$, the problem \eqref{ABC} reduces to a single maximally monotone inclusion, which can be solved by the proximal point algorithm \cite{proxi}.
	If the number of maximally monotone operators increases to two, i.e., $n=2$, $m=l=0$, the problem \eqref{ABC} becomes that of finding a zero of the sum of two maximally monotone operators.
	A classical approach for this problem is the Douglas--Rachford splitting method \cite{DR}, and it can also be interpreted as a special case of the proximal point algorithm \cite{DRPR}.
	For the case of three maximally monotone operators, i.e., $n=3$, $m=l=0$, Ryu \cite{Ryu} proposed a resolvent splitting method.
	This line of research was further extended by Malitsky and Tam \cite{MKT}, who introduced a resolvent splitting method with minimal lifting for the sum of $n$ maximally monotone operators. Tam \cite{Tam} further proposed a general framework for frugal and decentralised resolvent splittings based on nonexpansive operators. At almost the same time, Bredies et al. \cite{Bredies-Graph2024} introduced another framework which is graph-based extensions of the Douglas--Rachford
	splitting.
	
	When $m\neq0$ or $l\neq0$, some splitting methods incorporate forward steps, which allow certain operators to be evaluated explicitly.
	For instance, when $n=m=1$, $l=0$, the inclusion consists of the sum of a maximally monotone operator and a cocoercive operator, which can be solved by the forward-backward algorithm \cite{n=m=1}.
	For the case $n=l=1$, $m=0$, Tseng developed the forward-backward-forward method \cite{Tseng}.
	
	Motivated by the above fundamental cases, various splitting schemes have been developed for more general multi-operator settings.
	For the case $m=n-1$, $l=0$, Arag\'{o}n--Artacho et al. \cite{ring} introduced a distributed forward-backward scheme, while Bredies et al. \cite{Bred}  proposed both parallel and sequential extensions of the Davis--Yin method based on the preconditioned proximal framework and the product-space reformulation \cite{Pierra}.
	Inspired by \cite{Bredies-Graph2024}, Arag\'{o}n--Artacho et al. \cite{fran} developed a graph-based algorithmic framework.
	More recently, \AA kerman et al. \cite{Anton} further proposed an averaged frugal splitting framework relying on the individual cocoercivity constants of the operators $B_i$ rather than a global one, which allows the number $m$ of cocoercive operators to be not necessarily equal to $n-1$.  At the same time, Dao et al. \cite{minh} also provided a unified framework  from a different perspective.
	
	Another important direction is to incorporate the forward-backward splitting methods with reflection terms in order to handle the case $n\geq3$, $l=n-2$, $m=0$. Arag\'{o}n--Artacho et al. \cite{ring} proposed distributed forward-reflected-backward methods.
	Related reflected Davis--Yin type methods were further generalized in \cite{DYR}.
	Dao et al. \cite{minh} firstly developed a unified framework with the reflected terms for the case $n\geq3$, $m=0$.
	
	These developments naturally motivate the study of more general frameworks that can simultaneously incorporate maximally monotone, cocoercive, and monotone Lipschitz continuous operators.
	For the problem \eqref{ABC} with $n\geq3$, $m=n-1$, $l=n-2$, \cite{FBWR} proposed three forward-backward splitting methods with reflection terms.
	However,  their convergence analyses were carried out separately. Furthermore, the stepsizes depend on the global cocoercivity constants and global Lipschitz constants, which is generally conservative.
	
	To provide a systematic perspective and a unified convergence analysis,  we propose a general algorithmic framework for solving the problem \eqref{ABC}  by extending the framework
	and techniques from \cite{Anton}.
	The proposed framework can recover the aforementioned methods  and design new methods  through suitable matrix selections. It generalizes the problem setting considered in \cite{FBWR}, where $n\geq3$, $m=n-1$, $l=n-2$, to the general form of the problem \eqref{ABC}.
	Moreover, our framework relies on the individual cocoercivity constants of the operators $B_i$ and the individual Lipschitz constants of the operators $C_i$, rather than on global ones.
	\vskip 1mm
	
	Our main contributions are as follows.
	\begin{itemize}
		\item[{\rm(a)}] The proposed framework provides a unified perspective for the design and convergence analysis of a class of splitting algorithms. It also recovers a variety of existing well-known algorithms as special cases through suitable choices of the underlying matrices.
		
		\item[{\rm(b)}] The proposed framework can utilize the Lipschitz constant of each monotone and Lipschitz continuous operator as well as the cocoercive constant of each cocoercive operator, instead of relying on the global Lipschitz constant or the global cocoercive constant and significantly enlarges the admissible range of stepsizes.
		
		\item[{\rm(c)}] We further develop a heuristic matrix-selection strategy inspired by \cite{minh}, which aims to make more effective use of the admissible stepsize range. Based on this strategy, a concrete algorithm is derived from the proposed framework. The numerical results illustrate the superior performance of the proposed algorithm by comparing with existing algorithms.
	\end{itemize}
	
	The structure of the paper is as follows. In Section \ref{Sect2}, we introduce the notations
	and main concepts. In Section \ref{Sect3}, we propose a forward-backward
	splitting algorithmic framework with reflection terms. Section \ref{Sect4} shows the weak convergence of
	our framework. In Section \ref{Sect5}, we present a numerical experiment.
	
	\section{Preliminaries}
	\label{Sect2}
	
	Let $\mathbb{N}=\{0,1,\dots\}$ be the set of natural numbers and $\mathbb{N}_+=\{1,2,\dots\}$ be the set of non-zero natural numbers. Throughout this paper,  we denote by $\mathcal{H}$ a real Hilbert space with inner product  $\langle \cdot, \cdot\rangle$ and induced norm $\|\cdot\|=\sqrt{\langle \cdot, \cdot\rangle}$.  Denote by $\Id$  the identity operator on $\mathcal{H}$. 
	Given a linear, self-adjoint, and strongly  positive  operator $V: \mathcal{H}\to \mathcal{H}$, we define $\langle x,y  \rangle_V =\langle x,Vy  \rangle $  and  $ \| x  \|_V=\sqrt{ \langle x,Vx  \rangle }  $ for $\forall x,y \in\mathcal{H}$.     We use $\omega_w(x^k) = \{x : \exists x^{k_j}\rightharpoonup x\}$ to denote the weak $\omega$-limit set of the sequence $\{x^k\}$.
	
	For a matrix $P\in\mathbb{R}^{n\times m}$, we denote by $P_{ij}$ its $(i,j)$ component. The transpose of matrix $P$ is  denoted as  $P^{\top}$.  When $m=n,$ we denote by $\slt(P)\in\mathbb{R}^{n\times n}$ the strictly  lower triangular matrix extracted from $P$, and  by   $\diag(P)\in\mathbb{R}^n$ the vector   extracted  from the main diagonal of $P$.  With a mild overload of notation, for $w\in\mathbb{R}^n$, we denote by $\diag(w)\in\mathbb{R}^{n\times n}$ the diagonal matrix with the diagonal being $w$. Given a vector $g = (g_1,\dots, g_n)$ with all non-zero elements, the Hadamard inverse of $g$, denoted by $g^{\odot (-1)}$, is the vector defined element-wise by $(g^{\odot (-1)})_i = \frac{1}{g_i} $ for  $i=1,\dots,n$.
	Given a matrix $P\in \mathbb{R} ^{n\times m}$,  we denote the Kronecker  product of $P$ and $\Id$ by
	$$
	\mathbf{P} =P\otimes \Id=\begin{bmatrix}
		P_{11}\Id&  P_{12}\Id & \cdots   & P_{1m}\Id\\
		P_{21}\Id & P_{22}\Id  & \cdots &  P_{2m}\Id\\
		\vdots  & \vdots  &\ddots   &\vdots  \\
		P_{n1}\Id &  P_{n2}\Id &\cdots   & P_{nm}\Id
	\end{bmatrix}.
	$$
	Note that $\mathbf{P}$ is a bounded linear operator from $\mathcal{H}^m$ to $\mathcal{H}^n$.

	\begin{fact}
		For all $a, b, c, d\in\mathcal{H}$, there holds	
		\begin{equation}
			\label{ceq}
			2\left \langle a-b,c-d \right \rangle =\left \| a-d \right \| ^2+\left \| b-c \right \|^2-\left \| a-c \right \|^2-\left \| b-d \right \|  ^2.
		\end{equation}
	\end{fact}	
	
	Let   $T: \mathcal{H} \rightarrow {\mathcal{H}}$ be an operator. Denote by $\Fix T$ be the set of fixed points of $T,$ i.e., $\Fix T=\{x\in \mathcal{H} \,:\, x=T(x)\}.$
	\begin{dfn}
		\rm
		An operator  $T: \mathcal{H} \rightarrow {\mathcal{H}}$ is said to be
		\begin{itemize}
			\item[(i)]  $L $-Lipschitz continuous, if there exists a constant $ L > 0$, such that
			$$
			\|T(x)-T(y)\|\leq  L \|x-y\|,\quad\forall x, y\in \mathcal{H},
			$$
			and nonexpansive if $L =1$;
			\item[(ii)] $\sigma$-cocoercive, if there exists a constant $\sigma > 0$, such that
			$$
			\langle T(x)-T(y),x-y \rangle \geq  \sigma\|T(x)-T(y)\|^2,\quad\forall x, y\in \mathcal{H}.
			$$
		\end{itemize}
	\end{dfn}
	\noindent
	By  Cauchy--Schwarz inequality, a $\sigma$-cocoercive operator is $\frac1\sigma$-Lipschitz continuous.
	
	\begin{lem}\label{lemma2.2}
		Let  $C_i: \mathcal{H}\to \mathcal{H} $ be  $L_i$-Lipschitz continuous, $i=1,\dots,l$  and $\mathbf{C}=(C_1,\dots,C_{l})$. Then it holds
		\begin{equation*}\label{CC}		\|\mathbf{C}(\mathbf{x})-\mathbf{C}(\mathbf{y})\|\le\|\mathbf{x}-\mathbf{y}\|_{\mathbf{L^2}},
			\quad\forall \mathbf{x},\mathbf{y}\in \mathcal{H}^{l},
		\end{equation*}
		where $L^2=\diag(L_1^2,\dots,L_l^2)$.
	\end{lem}
	
	\begin{proof} Let $\mathbf{x}=(x_1,\ldots,x_l)$ and $\mathbf{y}=(y_1,\ldots,y_l)$ where $x_i,y_i\in\mathcal{H}$, $i=1,\ldots,l$. Then it follows from the Lipschitz continuous property of $C_i$ that
		\begin{equation*}
			\begin{aligned}
				\|\mathbf{C}(\mathbf{x})-\mathbf{C}(\mathbf{y})\|^2&={\sum_{i=1}^{l}\|C_ix_i-C_iy_i\|^2}\\				
				&\le{\sum_{i=1}^{l}L_i^2\|x_i-y_i\|^2}\\
				&={\sum_{i=1}^{l}\langle L_i^2(x_i-y_i),x_i-y_i\rangle}\\
				&={\langle \boldsymbol{L^2}(\mathbf{x}-\mathbf{y}),\mathbf{x}-\mathbf{y}\rangle}
				={\|\mathbf{x}-\mathbf{y}\|_{\boldsymbol{L^2}}^2}\\
			\end{aligned}
		\end{equation*}	
		The proof is completed.
	\end{proof}
	
	\begin{lem}\label{lemma2.1}
		Let $B_i: \mathcal{H}\to \mathcal{H} $ be $\sigma_i$-cocoercive,  $i=1,\dots,m$ and    $\mathbf{B}=(B_1,\dots,B_{m})$. Then it holds
		\begin{equation*}\label{BB}
			\langle \mathbf{B}(\mathbf{x})-\mathbf{B}(\mathbf{y}),\mathbf{x}-\mathbf{y}\rangle\ge
			\|\mathbf{B}(\mathbf{x})-\mathbf{B}(\mathbf{y})\|^2_{\mathbf{\Sigma}},
			\quad\forall \mathbf{x},\mathbf{y}\in \mathcal{H}^{m},
		\end{equation*}
		where   $\Sigma=\diag(\sigma_1,\dots,\sigma_m)$.
	\end{lem}
	\begin{proof} Let $\mathbf{x}=(x_1,\ldots,x_m)$ and $\mathbf{y}=(y_1,\ldots,y_m)$ where $x_i,y_i\in\mathcal{H}$, $i=1,\ldots,m$. Then the cocoercive property of $B_i$ implies
		\begin{equation*}
			\begin{aligned}
				\langle \mathbf{B}(\mathbf{x})-\mathbf{B}(\mathbf{y}),\mathbf{x}-\mathbf{y}\rangle&=
				\sum_{i=1}^{m}\langle B_ix_i-B_iy_i,x_i-y_i\rangle\\
				&\ge\sum_{i=1}^{m}\sigma_i\|B_ix_i-B_iy_i\|^2\\
				&=\sum_{i=1}^{m}\langle \sigma_i(B_ix_i-B_iy_i),B_ix_i-B_iy_i\rangle\\
				&=\langle \mathbf{\Sigma}(\mathbf{B}(\mathbf{x})-\mathbf{B}(\mathbf{y})),
				\mathbf{B}(\mathbf{x})-\mathbf{B}(\mathbf{y})\rangle	=\|\mathbf{B}(\mathbf{x})-\mathbf{B}(\mathbf{y})\|^2_{\mathbf{\Sigma}}.
			\end{aligned}
		\end{equation*}	
		The proof is completed.
	\end{proof}

	\begin{dfn}	\rm
		An operator $T:\mathcal{H}\to \mathcal{H}$ is said to be $\delta$-strongly quasi-nonexpansive if $\Fix T\ne \emptyset$ and there exists $\delta>0$ such that
		$$
		\|T(x)-y\|^2+\delta\|(\Id-T)(x)\|^2\le\|x-y\|^2,\quad \forall x\in\mathcal{H},\,\, \forall y\in \Fix T.
		$$
	\end{dfn}

	A sequence $\{z^k\}$ in $\mathcal{H}$ is said to be Fej\'{e}r monotone with respect to a nonempty set $\Theta$ of $\mathcal{H}$ if, for all $v\in \Theta$ and all $k \in \mathbb{N}$, there holds
	\begin{equation*}
		\|z^{k+1}-v\|\leq\|z^k-v\|.
	\end{equation*}
	
	Employing the proof of \cite[Proposition 2.2]{minh}, it is easy to show the following lemma.
	
	\begin{lem}\label{per1}{\rm(Krasnosel'ski\v{\i}--Mann iterations)}
		Let $T:\mathcal{H}\to \mathcal{H}$ be  $\delta$-strongly quasi-nonexpansive. Let $z^0\in\mathcal{H}$ and   set
		\begin{equation*}
			z^{k+1}=(1-\lambda_k)z^k+\lambda_kTz^k,\quad \forall k \in \mathbb{N}
		\end{equation*}
		where $\{\lambda_k\}$ is a sequence in $[0,1+\delta]$ such that $\lim\inf_{k\rightarrow\infty}\lambda_k(1-\lambda_k + \delta)>0$. Then the following hold:
		\begin{itemize}
			\item[\rm(i)] $\{z^k\}$ is Fej\'{e}r monotone with respect to $\Fix T$.
			\item[\rm(ii)]$(\Id-T)(z^k)\to0$ as $k\to\infty$.
			\item[\rm(iii)] $\|\frac{1}{k+1}\sum_{s=0}^{k}(\Id-T)(z^s)\|=O(\frac{1}{\sqrt{k}})$ as $k\to \infty$.
		\end{itemize}
	\end{lem}

	Given a set-valued operator $A:\mathcal{H}\rightrightarrows\mathcal{H}$, the domain, the range,  the graph and the zeros of $A$ are respectively denoted by
	$\dom A =\{x\in\mathcal{H} : A(x)\neq\varnothing\}$,  $\range A =\{u\in\mathcal{H}: u\in A(x)\ \hbox{for}\ \forall x\in\dom A\}$,
	$\gra A =\{(x,u)\in\mathcal{H}\times\mathcal{H} : u\in A(x)\}$  and
	$\zer A =\{x\in\mathcal{H} : 0\in A(x)\}$.
	The inverse operator of $A$, denoted by $A^{-1}$, is defined through $x\in A^{-1}(u)\Leftrightarrow u\in A(x)$.
	
	\begin{dfn}{\rm(\cite[Definition 20.1 and Definition 20.20]{BC2011})}
		{\rm
			A set-valued operator $A: \mathcal{H} \rightrightarrows\mathcal{H}$ is said to be
			\begin{itemize}
				\item[(i)]  monotone if $\langle x-y,u-v\rangle \geq 0$, $\forall (x,u),(y,v) \in \gra A $.
				\item[(ii)]  maximally monotone if there exists no monotone operator $B: \mathcal{H} \rightrightarrows\mathcal{H}$ such that $\gra B$ properly contains $\gra A,$ i.e., for every $(x,u) \in \mathcal{H}\times\mathcal{H}$
				$$
				(x,u) \in \gra A  \ \ \Leftrightarrow \ \ \langle x-y,u-v \rangle \geq 0,  \ \  \forall(y,v)\in \gra A.
				$$
			\end{itemize}
		}
	\end{dfn}
	
	Given an operator $A : \mathcal{H}\rightrightarrows\mathcal{H}$, the resolvent of $A$ with parameter $\lambda>0$ is  denoted by $J_{\lambda A} =(\Id+\lambda A)^{-1}$. From \cite[Corollary 23.11]{BC2011}, if the operator $A$ is  maximally  monotone, then $J_{\lambda A}$ is single-valued and 1-cocoercive.

	\begin{dfn}\rm($(m,n)$-nondecreasing vector).
		Let $n\in \mathbb{N}_+$, $m\in \mathbb{N}$. A vector
		$$
		E=(E_1,E_2,\dots,E_n)\in\{0,1,\dots,m\}^n
		$$
		is said to be $(m,n)$-nondecreasing  if $E_1=0,$ $E_n=m$ and $E_i\le E_{i+1}$ for each $i\in[1,n-1]$.
	\end{dfn}
	Let $n\in \mathbb{N}_+$, $m\in \mathbb{N}$ and $E$ be an $(m,n)$-nondecreasing  vector. Denote by $\mathcal{S}(E)$  the set of matrices $R\in \mathbb{R}^{n\times m}$ that have a  staircase structure w.r.t. $E$, i.e.,
	$$R_{ij}
	=0\ \ \hbox{for}\ \hbox{all} \ i=1,\dots,n\ \hbox{and}\ j>E_i.$$
	Conversely, $\mathcal{S}^c(E)$ denotes the set of matrices $R\in\mathbb{R}^{n\times m}$ that have a complement staircase structure w.r.t. $E$, i.e., $R_{ij}=0$ for all $i=1,\dots,n$ and $j\le E_i$.
	
	We 	next use the concepts of staircase and complement staircase to define the causal pair of matrices and the  relatively causal triple of matrices.
	\begin{dfn}\label{causal}\rm
		(Causal pair of matrices). A pair of matrices $H,G^{\top}\in \mathbb{R}^{n\times m}$ is said to be causal if there exists an $(m,n)$-nondecreasing  vector $E$ such that
		$$H\in\mathcal{S}(E)\ \hbox{and}\ G^\top\in\mathcal{S}^c(E).$$
	\end{dfn}
	
	\begin{dfn}\label{dfn2.6}
		\rm(Relatively causal triple of matrices).
		Let $n\in \mathbb{N}_+$, $m\in \mathbb{N}$. The matrices $P,Q,R^{\top}\in \mathbb{R}^{n\times m}$ are said to be  relatively causal if there exists a pair of $(m,n)$-nondecreasing  vectors $E$ and $F$ satisfying $E_{i-1}\geq F_i$ for each $i\in[2,n]$, such that
		$$
		Q\in\mathcal{S}(F),\,\ R^\top\in\mathcal{S}^c(E)\hbox{~and~}P\in\mathcal{S}^c(F)\cap\mathcal{S}(E).
		$$
	\end{dfn}

	\begin{rem}
		A detailed analysis of the several selection ways of causal pairs of matrices is presented in \cite{Anton}, which also covers all the selection approaches listed in \cite{minh}, thus, no further elaboration is provided here.
		
		To verify the validity of Definition \ref{dfn2.6}, we analyze the triple matrices $P,Q$ and $R$ in \cite[Example 4.3]{minh}, which is given as
		follows:
		\begin{equation*}
			P = \begin{bmatrix}
				\begin{array}{c}
					0_{1 \times (n-2)} \\
					\hline
					\mathrm{Id}_{(n-2)} \\
					\hline
					0_{1 \times (n-2)}
				\end{array}
			\end{bmatrix}, \quad
			Q = \begin{bmatrix}
				\begin{array}{c}
					0_{2 \times (n-2)} \\
					\hline
					\mathrm{Id}_{(n-2)}
				\end{array}
			\end{bmatrix}, \quad
			R = \begin{bmatrix} \mathrm{Id}_{(n-2)} \mid 0_{(n-2) \times 2} \end{bmatrix}.
		\end{equation*}
		Set $E=(0,1,...,n-2,n-2)^n$ and $F=(0,0,1,...,n-2)^n$ with $E_{i-1}= F_i$ for each $i\in[2,n]$. It is obvious that
		$Q\in\mathcal{S}(F),$ $R^\top\in \mathcal{S}^c(E)$ and
		$P\in\mathcal{S}^c(F)\cap\mathcal{S}(E)$.
		
	\end{rem}
	
	\begin{lem}\label{lle}{\rm(\cite[Lemma 2.47]{BC2011})}
		Let $\Theta$ be a nonempty set of $\mathcal{H}$, and $\{x^k\}$ be a sequence in $\mathcal{H}$.  Assume that the following conditions hold:\vskip 1mm
		\begin{itemize}
			\item[{\rm(i)}]  for every $x \in \Theta$, $\lim_{k \rightarrow \infty}\|x^{k}-x\|$ exists;
			\noindent
			\item[{\rm(ii)}] every weak sequential cluster point of $\{x^k\}$ belongs to $\Theta$, i.e., $\omega_w(x^k)\subseteq \Theta$.
		\end{itemize}
		Then the sequence $\{x^k\}$ converges weakly to a point in $\Theta$.
	\end{lem}

	\section{Reflected forward-backward splitting algorithm}
	\label{Sect3}
	
	In this section, we first develop a reflected forward-backward splitting  algorithmic framework for solving the problem \eqref{ABC}. Then we present the relation of the proposed algorithm and some existing
	methods.
	
	Throughout this section, the solution set
	of the problem \eqref{ABC} is assumed to be nonempty, i.e., $\zer( {\textstyle \sum_{i=1}^{n}}  A_i+{\textstyle \sum_{i=1}^{m}} B_i+\sum_{i=1}^{l})\neq \emptyset $. Let $B_i: \mathcal{H}\to \mathcal{H} $ be $\sigma_i$-cocoercive,  $i=1,\dots,m$ and $C_i: \mathcal{H}\to \mathcal{H} $ be monotone and $L_i$-Lipschitz continuous, $i=1,\dots,l$. Furthermore, let $\Sigma=\diag(\sigma_1,\dots,\sigma_m)$ and $L=\diag(L_1,\dots,L_l)$.
	
	\subsection{Algorithm}
	Before giving our algorithm, we  introduce some matrices satisfying the following assumptions.
	\begin{asm}
		\label{asm11}
		The  matrices $M\in \mathbb{R}^{n\times (n-1)},  H,  G^{\top}\in\mathbb{R}^{n\times m}, P,Q,R^\top\in \mathbb{R}^{n\times l}$ and  $K\in \mathbb{R}^{n\times n}$ satisfy the following properties:
		\begin{itemize}
			\item[{\rm(a)}]$\ker M^{\top}\subseteq \mathbb{R}e_n$, where $e_n=(1,\dots,1)^{\top}\in \mathbb{R}^n$.
			\item[{\rm(b)}]$H^{\top}e_n=Ge_n=e_m,$ $Q^{\top}e_n=P^{\top}e_n=Re_n=e_l,$
			where $e_m=(1,\dots,1)^{\top}\in\mathbb{R}^m$ and $e_l=(1,\dots,1)^{\top}\in\mathbb{R}^l$.
			$H$, $G^\top$ are causal, and $P,Q,R^{\top}$ are relatively causal.
			\item[{\rm(c)}]$K$ is symmetric   such that $e_n^{\top}Ke_n=0$ and $K-MM^{\top}-\frac{1}{2}(H-G^{\top}){\Sigma^{-1}}(H^{\top}-G)-
			(P-Q)L(P^\top-Q^\top)-(P-R^{\top})L(P^\top-R)\succeq 0$.
		\end{itemize}
	\end{asm}

	Now we present a reflected forward-backward splitting algorithmic framework for solving
	the problem (\ref{ABC}).
	
	\begin{alg}\label{AlgG}
		\hrule\vskip 0.5mm
		\rm\noindent\rm
		\vskip 0.5mm
		\hrule
		
		\vskip 1mm
		\noindent
		\textbf{Pick:} $\gamma_k\in(0,1)$, and the matrices $M\in \mathbb{R}^{n\times (n-1)},  H,  G^{\top}\in\mathbb{R}^{n\times m}, P,Q,R^\top\in \mathbb{R}^{n\times l},$
		$K\in \mathbb{R}^{n\times n}$ satisfying Assumption \ref{asm11} and $U\in\mathbb{R}^{n\times t},t\geq1$.\\
		\textbf{Let:}\\ \small$K=MM^{\top}+UU^{\top}+\frac{1}{2}(H-G^{\top}){\Sigma^{-1}}(H^{\top}-G)+(P-Q)
		\normalsize
		L(P^\top-Q^\top)+(P-R^\top)L(P^\top-R)$ and $(d_1,\dots,d_n)=2(\diag(K))^{\odot(-1)}$.\\
		1: \textbf{Input:} $\mathbf{z}^0=(z_1^0,\dots,z_{n-1}^0)\in \mathcal{H}^{n-1}$.\\
		2: \textbf{for}  $k=0,1,2,\dots$ \textbf{do} \\
		3: \ \ \ \ \  \textbf{for} $i=1,\dots,n$ \textbf{do}\\
		\begin{equation}
			\label{main-scheme}
			\begin{aligned}
				x_i^{k} &= J_{d_i A_i}\biggl( -d_i\sum_{j=1}^{i-1}K_{ij}x_j^{k} + d_i\sum_{j=1}^{n-1}M_{ij}z_j^k - d_i\sum_{j=1}^{m}H_{ij} B_{j}\biggl( \sum_{h=1}^{i-1} G_{jh}x_h^k \biggr) \\
				&\quad - d_i\sum_{j=1}^{l}(P-Q)_{ij} C_{j}\biggl( \sum_{h=1}^{i-1} R_{jh}x_h^k \biggr) - d_i\sum_{j=1}^{l}Q_{ij} C_{j}\biggl( \sum_{h=1}^{i-1} P_{hj}x_h^k \biggr) \biggr).
			\end{aligned}
		\end{equation}
		4: \ \ \ \ \  \textbf{end for}\\
		5:  $z_i^{k+1}=z_i^k-\gamma_k{\textstyle \sum_{i=1}^{n}} M_{ij}x_i^k,
		\quad i \in\{1,\ldots,n-1\}.$\\
		6: \textbf{end for}
		\vskip 1mm
		
		\hrule
		
		\hspace*{\fill}
	\end{alg}

	Let
	$\mathbf{A}=(A_1,...,A_n)$ and $D=\diag(d_1,\dots,d_n)$. It follows from \cite[Proposition 20.23]{BC2011} that $\mathbf{A}$ is  maximally monotone  and its resolvent $J_{\mathbf{D}\mathbf{A}}: \mathcal{H}^n\to \mathcal{H}^n$ is given by $J_{\mathbf{D}\mathbf{A}}=(J_{d_1A_1}, \dots,J_{d_nA_n})$. Similarly let
	$\mathbf{B}=(B_1,...,B_{m})$  and
	$\mathbf{C}=(C_1,...,C_{l}).$
	Then  Algorithm \ref{AlgG} can be recast as
	\begin{equation*}\label{T1}
		\left\{
		\aligned
		&\mathbf{x}^k=J_{ \mathbf{ D A}}(\mathbf{DNx}^k+\mathbf{DMz}^k-\mathbf{DHB(Gx}^k)
		-\mathbf{D(P-Q)C(Rx}^k)-\mathbf{DQC(P^*x}^k)),\\
		&	\mathbf{z}^{k+1}= \mathbf{z}^k-\gamma_k\mathbf{M}^*\mathbf{x}^k,
		\endaligned
		\right.
	\end{equation*}
	where $N=-\slt(K)$.
	For the convenience of the convergence analysis, we rewrite the above scheme as
	\begin{equation*}
		\label{Td1}
		\mathbf{z}^{k+1}=(1-\frac{\gamma_k}{\gamma})\mathbf{z}^k+\frac{\gamma_k}{\gamma}T\mathbf{z}^k,
	\end{equation*}
	where $\gamma>0$ and the	operator $T:\mathcal{H}^{n-1}\to \mathcal{H}^{n-1}$ is defined by
	\begin{equation}\label{T}
		\aligned
		T(\mathbf{z})= \mathbf{z}-\gamma\mathbf{M}^*\mathbf{x},
		\endaligned
	\end{equation}
	where $\mathbf{x}=J_{ \mathbf{ D A}}\mathbf{(DNx+DMz-DHB(Gx)}
	-\mathbf{D(P-Q)C(Rx)-DQC(P^*x))}.$
	
	\begin{rem}
		\label{rem11}
		\begin{itemize}	\item[{\rm(1)}]
			According to \cite{Anton},	a simple choice of $K$ satisfying Assumption \ref{asm11}(c) is
			$
			K=MM^{\top}+UU^{\top}+\frac{1}{2}(H-G^{\top}){\Sigma^{-1}}(H^{\top}-G)+(P-Q)
			L(P^\top-Q^\top)+(P-R^\top)L(P^\top-R),
			$
			where $U\in\mathbb{R}^{n\times t},t\geq1$ is such that $\mathbb{R}e_n\subseteq \ker U^{\top}$. Since $d_i$, $i=1,\ldots,n$ are generally regarded as  the stepsizes in Algorithm \ref{AlgG}, larger stepsizes can be obtained by setting $U=0$ within $K$.
			\item[{\rm(2)}]
			By $N=-\slt(K)$ and $D=\diag(d_1,\dots,d_n)$, we get $K=2D^{-1}-N-N^{\top}$.
			Then Assumption \ref{asm11}(c) becomes that	the matrices $M\in \mathbb{R}^{n\times (n-1)},  H,  G^{\top}\in\mathbb{R}^{n\times m}, P, Q, R^\top\in \mathbb{R}^{n\times l}$ and  $N, D\in \mathbb{R}^{n\times n} $ satisfy the following properties:
			\begin{itemize}
				\item[{\rm(a)}] $e_n^{\top}(D^{-1}-N)e_n=0$;
				\item[{\rm(b)}]$2D^{-1}-N-N^{\top}\succeq MM^{\top}+\frac{1}{2}(H-G^\top){\Sigma^{-1}}(H^{\top}-G)+(P-Q) L(P^\top-Q^\top)+(P-R^\top) L(P^\top-R)$.
			\end{itemize}
			\item[{\rm(3)}]
			The concepts of causal pairs and relatively causal triple of matrices captures the full class of matrices $H$, $G^{\top}\in\mathbb{R}^{n\times m}$, $P$, $Q$, $R^\top$ $\in\mathbb{R}^{n\times l}$, such that the terms $\mathbf{HB(Gx)}$, $\mathbf{(P-Q)C(Rx)}$ and $\mathbf{QC(P^*x)}$ in  \eqref{T} can be computed by using only a single evaluation of each operator $B_i$ and $C_i$,  together with simple algebraic operations. Furthermore, $\mathbf{HB(Gx)}$, $\mathbf{(P-Q)C(Rx)}$ and $\mathbf{QC(P^*x)}$ are strictly lower triangular for all choices of $B_i$ and $C_i$ according to Definitions \ref{causal} and \ref{dfn2.6}. Therefore, the proposed algorithmic framework is explicit.
		\end{itemize}
	\end{rem}
	
	We can present a lifted algorithm reformulation as  in \cite{Anton}. Let $\mathcal{L}=MM^\top$ and $\mathbf{w}^k=\mathbf{M}\mathbf{z}^k$, then Algorithm \ref{AlgG} with $\mathbf{z}^0=0$ is equivalent to the following form,
	\begin{equation}\label{Simp}
		\left\{
		\begin{aligned}
			&x_i^{k} = J_{d_i A_i}\biggl( -d_i\sum_{j=1}^{i-1}K_{ij}x_j^{k} + d_i\sum_{j=1}^{n}w_j^k - d_i\sum_{j=1}^{m}H_{ij} B_{j}\biggl( \sum_{h=1}^{i-1} G_{jh}x_h^k \biggr) \\
			& - d_i\sum_{j=1}^{l}(P-Q)_{ij} C_{j}\biggl( \sum_{h=1}^{i-1} R_{jh}x_h^k \biggr) - d_i\sum_{j=1}^{l}Q_{ij} C_{j}\biggl( \sum_{h=1}^{i-1} P_{hj}x_h^k \biggr) \biggr),i \in\{1,\ldots,n\},\\
			&w_i^{k+1}=w_i^k-\gamma_k{\textstyle \sum_{i=1}^{n}} \mathcal{L}_{ij}x_i^k, \quad\forall i \in\{1,\ldots,n\},
		\end{aligned}
		\right.
	\end{equation}
	where the initial point $\mathbf{w}^0 =0$.
	The above scheme eliminates one matrix operation  each iteration while it sacrifices the minimum enhancement property of Algorithm \ref{AlgG} due to $\mathbf{w}^k\in \mathcal{H}^n$.

	\subsection{Relationship with some existing methods}
	
	In form, the proposed Algorithm \ref{AlgG} is a  forward-backward type splitting algorithmic framework with reflected terms. When $n\ge 1$, $m\ge 1$ and $l\ge 1$ in the problem \eqref{ABC}, our proposed Algorithm \ref{AlgG} is new to the literature.
	
	Algorithm \ref{AlgG} reduces   to previous work when  the matrices $P,Q,R,H,G$ are specially selected.
	\begin{itemize}
		\item[(i)] In the case of  $P=Q=R^\top=0$,  Algorithm \ref{AlgG} becomes to
		\begin{equation}
			\label{AA}
			\left\{
			\begin{aligned}
				&x_i^{k} = J_{d_i A_i}\biggl( -d_i\sum_{j=1}^{i-1}K_{ij}x_j^{k} + d_i\sum_{j=1}^{n-1}M_{ij}z_j^k - d_i\sum_{j=1}^{m}H_{ij} B_{j}\biggl( \sum_{h=1}^{i-1} G_{jh}x_h^k \biggr)\biggl),\\
				&\qquad\qquad\qquad\qquad\qquad\qquad\qquad\qquad\qquad\qquad i\in\{1,\ldots,n\},\\
				&z_i^{k+1}=z_i^k-\gamma_k{\textstyle \sum_{i=1}^{n}} M_{ij}x_i^k,
				\quad i \in\{1,\ldots,n-1\},
			\end{aligned}
			\right.
		\end{equation}
		which solves the inclusion problem $0 \in \sum_{i=1}^{n}A_i(x)+\sum_{i=1}^{m}B_i(x)$. \AA kerman et al. \cite{Anton} first introduced the algorithmic framework \eqref{AA}. Then Dao et al. \cite{minh} further analyzed this framework under the assumption $\sigma_1=\cdots=\sigma_m$.
		
		\item[(ii)] In the case of $H=G^\top=0$, Algorithm \ref{AlgG} becomes to
		\begin{equation}
			\label{AA1}
			\left\{
			\begin{aligned}
				&x_i^{k} = J_{d_i A_i}\biggl( -d_i\sum_{j=1}^{i-1}K_{ij}x_j^{k} + d_i\sum_{j=1}^{n-1}M_{ij}z_j^k - d_i\sum_{j=1}^{l}(P-Q)_{ij} C_{j}\biggl( \sum_{h=1}^{i-1} R_{jh}x_h^k \biggr)\\
				& \qquad- d_i\sum_{j=1}^{l}Q_{ij} C_{j}\biggl( \sum_{h=1}^{i-1} P_{hj}x_h^k \biggr) \biggr),\quad i\in\{1,\ldots,n\},\\
				&z_i^{k+1}=z_i^k-\gamma_k{\textstyle \sum_{i=1}^{n}} M_{ij}x_i^k,
				\quad i \in\{1,\ldots,n-1\},
			\end{aligned}
			\right.
		\end{equation}
		which solves the inclusion problem $0 \in   \sum_{i=1}^{n}A_i(x)+\sum_{i=1}^{l}C_i (x)$.
		Dao et al. \cite{minh} first proposed the algorithmic framework \eqref{AA1} and established its  convergence under the assumption  $L_1=\cdots=L_l$.
		
		\item[(iii)]	Algorithm \ref{AlgG} encompass there splitting methods in  \cite{FBWR}.
		The appendix provides detailed illustrations. Furthermore, it is worth mentioning that the proposed framework may yield a larger stepsize range than a direct analysis of the corresponding individual algorithm, as shown in Appendix \ref{a2}.
	\end{itemize}

	\section{Convergence analysis}	
	\label{Sect4}
	
	In this section the weak convergence of Algorithm \ref{AlgG} is established.
	
	The next lemma gives the relation of the fixed point set of $T$ and the  zeros of ${\textstyle \sum_{i=1}^{n}}  A_i+{\textstyle \sum_{i=1}^{m}} B_i+\sum_{i=1}^{l} C_i$.	
	
	\begin{lem}
		\label{lem11}
		(Fixed points and zeros)
		Let
		$$
		\aligned
		\Omega=\{(\mathbf{z},x)\in \mathcal{H}^{n-1}\times \mathcal{H}:\ &\mathbf{x}=J_{ \mathbf{DA}}(\mathbf{DNx+DMz- DHB(Gx)-DPC(Rx)}) \,\,\hbox{where}\\
		&\mathbf{x}=(x,\dots,x)\in\mathcal{H}^n\}.
		\endaligned
		$$
		Suppose that Assumption \ref{asm11}(a)$\&$(b) hold. Then the following assertions hold.
		\begin{itemize}
			\item[{\rm(a)}] If $\mathbf{z}\in \Fix T$, then there exists $x\in \mathcal{H}$ such that $(\mathbf{z},x)\in \Omega$.
			\item[{\rm(b)}] If $x\in \zer \left (  {\textstyle \sum_{i=1}^{n}} A_i+ {\textstyle \sum_{i=1}^{m}B_i} + {\textstyle \sum_{i=1}^{l}C_i} \right ) $, then there exists $\mathbf{z}\in \mathcal{H}^{n-1}$ such that $(\mathbf{z},x)\in \Omega$.
			\item[{\rm(c)}] If $(\mathbf{z},x)\in \Omega$, then $\mathbf{z}\in \Fix T $ and $x\in \zer \left (  {\textstyle \sum_{i=1}^{n}} A_i+ {\textstyle \sum_{i=1}^{m}B_i}  + {\textstyle \sum_{i=1}^{l}C_i} \right ) $.
			Consequently, $$\Fix T \ne \emptyset \Leftrightarrow \Omega \ne \emptyset \Leftrightarrow \zer \left (  {\textstyle \sum_{i=1}^{n}} A_i+ {\textstyle \sum_{i=1}^{m}B_i}  + {\textstyle \sum_{i=1}^{l}C_i} \right )\ne \emptyset.$$
		\end{itemize}
	\end{lem}
	\begin{proof}
		(a): Let $\mathbf{z}\in \Fix T$ and set $\mathbf{x}=J_{ \mathbf{DA}}(\mathbf{DNx+DMz- DHB(Gx)-DPC(Rx)})$. Since $T(\mathbf{z})=\mathbf{z}-\gamma\mathbf{M^*x}$ and $\mathbf{z}\in\Fix T$, we have $\mathbf{x}=(x,\dots,x)$.
		
		(b): Let $x\in\zer \left (  {\textstyle \sum_{i=1}^{n}} A_i+ {\textstyle \sum_{i=1}^{m}B_i}  + {\textstyle \sum_{i=1}^{l}C_i} \right )$, and set $\mathbf{x}={(x,\dots,x)}.$ Then there exists $\mathbf{v}=(v_1,\dots,v_n)\in\mathcal{H}^n$ such that $v_i\in A_i(x)$, and ${\textstyle \sum_{i=1}^{n}} v_i+ {\textstyle \sum_{i=1}^{m}B_i} + {\textstyle \sum_{i=1}^{l}C_i}=0$. Define $\mathbf{y}=(y_1,\dots,y_n)\in\mathcal{H}^n$ according to $\mathbf{y=Dv+x}$, so that $\mathbf{y\in DAx+x}$ and $\mathbf{x}=J_{\mathbf{DA}}(\mathbf{y})$. To complete the proof, we must show there exists $\mathbf{z}\in$$\mathcal{H}^{n-1}$ such that $\mathbf{y=\mathbf{DNx+DMz- DHB(Gx)-DPC(Rx)}}$, which is equivalent to the $\mathbf{D^{-1}y-Nx+HB(Gx)+PC(Rx)}\in \range \mathbf{M}$. To this end, first note that Assumption \ref{asm11} (a) and (c) imply
		\begin{equation*}
			\range\mathbf{M}=(\ker\mathbf{M^*})^\perp=\{(x_1,...,x_n)\in\mathcal{H}^n:\sum_{i=1}^{n}x_i=0\}.
		\end{equation*}
		By Assumption \ref{asm11}(b)\&(c) and (a) of Remark \ref{rem11} (2), we have
		\begin{equation*}
			\begin{aligned}
				&\sum_{i=1}^{n}(\mathbf{D}^{-1}\mathbf{y-Nx+HB(Gx)+PC(Rx)})_i\\
				=&\sum_{i=1}^{n}\frac{1}{d_i}y_i-\sum_{i=1}^{n}\sum_{j=1}^{n}N_{ij}x+\sum_{i=1}^{n}\sum_{j=1}^{m}H_{ij}(B_j(\sum_{k=1}^{n}G_{jk}x))+\sum_{i=1}^{n}\sum_{j=1}^{l}P_{ij}(C_j(\sum_{k=1}^{n}R_{jk}x))\\				=&\sum_{i=1}^{n}v_i+\sum_{i=1}^{n}\frac{x}{d_i}-\sum_{i=1}^{n}\sum_{j=1}^{n}N_{ij}x+\sum_{i=1}^{n}\sum_{j=1}^{m}H_{ij}(B_j(\sum_{k=1}^{n}G_{jk}x))\\
				&+\sum_{i=1}^{n}\sum_{j=1}^{l}P_{ij}(C_j(\sum_{k=1}^{n}R_{jk}{x}))\\
				=&\sum_{i=1}^{n}v_i+\sum_{j=1}^{m}B_j(x)+\sum_{j=1}^{l}C_j(x)=0.\\
			\end{aligned}
		\end{equation*}
		Hence $\mathbf{D^{-1}y-Nx+HB(Gx)+PC(Rx)}\in\range \mathbf{M}$, as required.
		
		(c): Let $(\mathbf{z},x)\in \Omega$ and set $\mathbf{y=\mathbf{DNx+DMz- DHB(Gx)-DPC(Rx)}}$ where $\mathbf{x}=(x,\dots,x),$ we have $\mathbf{z}\in\Fix T$.  Since $\mathbf{x=}J_{\mathbf{DA}}(\mathbf{y})$, we have $\mathbf{DA(x)}\ni\mathbf{y-x}=\mathbf{DNx+DMz- DHB(Gx)-DPC(Rx)}-\mathbf{x}.$ Then $\mathbf{A(x)+HB(Gx)}+\mathbf{PC(Rx)}\ni\mathbf{ Nx+Mz-D^{-1}x}$, Assumption \ref{asm11}(a)\&(b) and (a) of Remark \ref{rem11}(2) give
		\begin{equation*}
			(\sum_{i=i}^{n}A_i+\sum_{i=1}^{m}B_i+\sum_{i=1}^{l}C_i)(x)\ni ((e^\top M)\otimes\Id)\mathbf{z}+\sum_{i=1}^{n}\sum_{j=1}^{n}N_{ij}x-\sum_{i=1}^{n}\frac{x}{d_i}=0,
		\end{equation*}
		which shows that $x\in \zer(\sum_{i=1}^{n}A_i+\sum_{i=1}^{m}B_i+\sum_{i=1}^{l}C_i)$.
	\end{proof}
	
	Now we present a lemma which is key for the  convergence analysis of Algorithm \ref{AlgG}.
	\begin{lem}\label{lem22}
		Suppose Assumption \ref{asm11}  holds.	Let $\bar{\mathbf{z}}=(\bar{z}_1,\dots,\bar{z}_{n-1})\in\Fix T$. Then, for all $\mathbf{z}=(z_1,\dots,z_{n-1})\in\mathcal{H}^{n-1}$, we have
		\begin{equation}\label{1.2}
			\begin{aligned}
				&\|T(\mathbf{z})-\mathbf{\bar{z}}\|^{2}+	\frac{1-\gamma}{\gamma}\|(\Id-T)(\mathbf{z})\|^2\leq
				\|\mathbf{z}-\mathbf{\bar{z}}\|^{2}\\
				&+\gamma\langle\mathbf{x}-\mathbf{\bar{x}},[\mathbf{M}\mathbf{M}^{*}+\mathbf{N}
				+\mathbf{N}^*-2\mathbf{D}^{-1}+\frac{1}{2}\mathbf{(H-G^{*}){\Sigma^{-1}}(H^{*}-G)}\\				&+\mathbf{(P-Q)\boldsymbol{L}(P^*-Q^*)+(P-R^{*})\boldsymbol{L}(P^*-R)}](\mathbf{x}-
				\mathbf{\bar{x}})\rangle.
			\end{aligned}
		\end{equation}
		In particular, if $\gamma\in (0,1)$, then $T$ is $\delta$-strongly quasi-nonexpansive with $\delta=\frac{1-\gamma}{\gamma}$.
	\end{lem}
	\begin{proof}
		Let $\mathbf{x}=J_{\mathbf{ DA}}(\mathbf{y})$ with	
		$\mathbf{y}=\mathbf{DNx+DMz-DHB(Gx)}-\mathbf{D(P-Q)}\mathbf{C(Rx)}-\mathbf{DQC(P^*x)}$ and $\mathbf{\bar{x}}=J_{\mathbf{ DA}}(\mathbf{\bar{y}})$ with
		$\mathbf{\bar{y}}=\mathbf{DN\bar{x}}+\mathbf{DM\bar{z}}-\mathbf{DHB(G\bar{x})}-\mathbf{D(P-Q)C(R\bar{x})}-\mathbf{DQC(P^*\bar{x})}$.  Since $\mathbf{D}^{-1}\mathbf{(y-x)}\in\mathbf{A}(\mathbf{x})$ and $\mathbf{D}^{-1}\mathbf{(\bar{y}}-\mathbf{\bar{x}})\in\mathbf{A}(\mathbf{\bar{x}})$, the monotonicity of $\mathbf{A}$ gives
		\begin{equation}\label{chang}
			\aligned
			0  \leq&\langle\mathbf{x}-\mathbf{\bar{x}},\mathbf{D}^{-1}(\mathbf{y}-\mathbf{x})-\mathbf{D}^{-1}(\mathbf{\bar{y}}-\mathbf{\bar{x}})\rangle \\
			=&\langle\mathbf{x}-\mathbf{\bar{x}},(\mathbf{Nx+Mz-HB(Gx)-(P-Q)C(Rx)-QC(P^*x)-D}^{-1}\mathbf{x}\rangle \\
			&-\langle\mathbf{x}-\mathbf{\bar{x}},(\mathbf{N\bar{x}+M\bar{z}-HB(G\bar{x})-(P-Q)C(R\bar{x})-QC(P^*\bar{x})-D}^{-1}\mathbf{\bar{x}}\rangle\\ =&\langle\mathbf{M}^{*}\mathbf{x}-\mathbf{M}^{*}\mathbf{\bar{x}
			},\mathbf{z}-\mathbf{\bar{z}}\rangle+\langle\mathbf{x}-\mathbf{\bar{x}},(\mathbf{N}-\mathbf{D}^{-1})\mathbf{x}-(\mathbf{N}-\mathbf{D}^{-1})\mathbf{\bar{x}}\rangle\\
			&-\left \langle \mathbf{x}-\mathbf{\bar{x}},\mathbf{HB(Gx)-HB(G\bar{x}}) \right \rangle\\
			&-\left \langle \mathbf{x}-\mathbf{\bar{x}},\mathbf{(P-Q)C(Rx)+QC(P^*x)-(P-Q)C(R\bar{x})-QC(P^*\bar{x}}) \right \rangle.\\
			\endaligned
		\end{equation}
		The first term on the RHS of  (\ref{chang}) can be expressed as
		\begin{equation}\label{101}
			\aligned
			\langle\mathbf{M}^{*}\mathbf{x}-\mathbf{M}^{*}\mathbf{\bar{x}}
			,\mathbf{z}-\mathbf{\bar{z}}\rangle
			=&\langle \frac{1}{\gamma}(\Id-T)(\mathbf{z})-\frac{1}{\gamma}(\Id-T)(\mathbf{\bar{z}}),
			\mathbf{z}-\mathbf{\bar{z}}\rangle \\
			=&\frac{1}{\gamma}\langle \mathbf{z}-T(\mathbf{z}),
			\mathbf{z}-\mathbf{\bar{z}}\rangle \\
			=&\frac{1}{2\gamma}\left(\|\mathbf{z}-\mathbf{\bar{z}}\|^{2}+\|(\Id-T)(\mathbf{z})\|^{2}-\|T(\mathbf{z})-\mathbf{\bar{z}}\|^{2}\right),
			\endaligned
		\end{equation}
		where the third equality comes from  \eqref{ceq}. The second term on the RHS of  (\ref{chang}) can be rewritten as
		\begin{equation}\label{102}
			\aligned
			&\langle\mathbf{x}  -\mathbf{\bar{x}},(\mathbf{N}-\mathbf{{D}}^{-1})\mathbf{x}-(\mathbf{N}-\mathbf{{D}}^{-1})\mathbf{\bar{x}}\rangle \\
			=&\frac{1}{2}\langle\mathbf{x}-\mathbf{\bar{x}},(\mathbf{M}\mathbf{M}^*+2\mathbf{N}-2\mathbf{D}^{-1})
			(\mathbf{x}-{\mathbf{\bar{x}}})\rangle-\frac{1 }{2}\|\mathbf{M}^*\mathbf{x}-\mathbf{M}^*\mathbf{\bar{x}}\|^2 \\	
			=&\frac{1}{2}\langle\mathbf{x}-\mathbf{\bar{x}},(\mathbf{M}\mathbf{M}^{*}+\mathbf{N}+\mathbf{N}^*-2\mathbf{D}^{-1})(\mathbf{x}-\mathbf{\bar{x}})\rangle -\frac{1}{2\gamma^2}\|(\Id-T)(\mathbf{z})\|^2.
			\endaligned
		\end{equation}
		By	Young's inequality  and Lemma \ref{lemma2.1},  the third term on the RHS of  (\ref{chang}) can be recast as
		\begin{equation}\label{103}
			\aligned
			&-\left \langle \mathbf{x}-\mathbf{\bar{x}},\mathbf{HB(Gx)-HB(G\bar{x}}) \right \rangle\\
			=&-\left \langle \mathbf{H^*}(\mathbf{x}-\mathbf{\bar{x}}),\mathbf{B(Gx)-B(G\bar{x}}) \right \rangle\\
			=&-\left \langle (\mathbf{H^*-G})(\mathbf{x}-\mathbf{\bar{x}}),\mathbf{B(Gx)-B(G\bar{x}})\right \rangle - \left \langle \mathbf{G}(\mathbf{x})-\mathbf{G}(\mathbf{\bar{x}}),\mathbf{B(Gx)-B(G\bar{x}}) \right \rangle\\
			\leq&\frac{1}{4}\|(\mathbf{H^*-G})(\mathbf{x}-\mathbf{\bar{x}})\|^2_{\mathbf{\Sigma}^{-1}}+\|\mathbf{B(Gx)-B(G\bar{x}})\|^2_{\mathbf{\Sigma}}-\|\mathbf{B(Gx)-B(G\bar{x}})\|^2_{\mathbf{\Sigma}}\\
			=&\frac{1}{4}\|(\mathbf{H^*-G})(\mathbf{x}-\mathbf{\bar{x}})\|^2_{\mathbf{\Sigma}^{-1}}
			=\frac{1}{4}\langle (\mathbf{H-G^*})\mathbf{\Sigma}^{-1}(\mathbf{H^*-G})(\mathbf{x}-\mathbf{\bar{x}}),\mathbf{x}-\mathbf{\bar{x}}\rangle.
			\endaligned
		\end{equation}
		By $\bar{\mathbf{z}}\in\Fix T$ and the definition of $\mathbf{\bar x}$,	it follows from Assumption \ref{asm11}(b) that $\mathbf{P^*\bar x=R\bar x}$.	Using  the monotonicity of $\mathbf{C}$, we can get
		\begin{equation}\label{middle}
			\begin{aligned}
				0&\leq\left \langle \mathbf{P^*}(\mathbf{x}-\mathbf{\bar{x}}),\mathbf{C(P^*x)-C(P^*\bar{x})}\right \rangle\\
				&=\left \langle \mathbf{P^*}(\mathbf{x}-\mathbf{\bar{x}}),\mathbf{C(P^*x)-C(Rx)}\right \rangle+\left \langle \mathbf{P^*}(\mathbf{x}-\mathbf{\bar{x}}),\mathbf{C(Rx)-C(P^*\bar{x})}\right \rangle\\
				&=\left \langle \mathbf{P^*}(\mathbf{x}-\mathbf{\bar{x}}),\mathbf{C(P^*x)-C(Rx)}\right \rangle+\left \langle \mathbf{P^*}(\mathbf{x}-\mathbf{\bar{x}}),\mathbf{C(Rx)-C(R\bar{x})}\right \rangle.
			\end{aligned}
		\end{equation}
		To estimate the last term on the RHS of  (\ref{chang}), combining Lemma \ref{lemma2.2} and \eqref{middle} gives
		\begin{equation}\label{104}
			\aligned
			&-\left \langle \mathbf{x}-\mathbf{\bar{x}},\mathbf{(P-Q)C(Rx)+QC(P^*x)-(P-Q)C(R\bar{x})-QC(P^*\bar{x}}) \right \rangle.\\	
			=&	\left \langle \mathbf{Q^*}(\mathbf{x}-\mathbf{\bar{x}}),
			\mathbf{C(Rx)-C(R\bar{x})} \right \rangle
			-\left \langle \mathbf{P^*}(\mathbf{x}-\mathbf{\bar{x}}),
			\mathbf{C(Rx)-C(R\bar{x})} \right \rangle\\
			&-\left \langle \mathbf{Q^*}(\mathbf{x}-\mathbf{\bar{x}}),
			\mathbf{C(P^*x)-C(P^*\bar{x})} \right \rangle\\
			\leq&	\left \langle \mathbf{Q^*}(\mathbf{x}-\mathbf{\bar{x}}),
			\mathbf{C(Rx)}-\mathbf{C(P^*x)}
			\right \rangle
			+\left \langle \mathbf{P^*}(\mathbf{x}-\mathbf{\bar{x}}),
			\mathbf{C(P^*x)-C(Rx)} \right \rangle
			\\
			=&	\left \langle (\mathbf{P^*-Q^*})(\mathbf{x}-\mathbf{\bar{x}}),\mathbf{C(P^*x)-C(Rx)} \right \rangle\\
			\leq& \frac{1}{2}\|(\mathbf{P^*-Q^*})(\mathbf{x}-\mathbf{\bar{x}})\|^2_{\boldsymbol{L}} + \frac{1}{2}\|\mathbf{C(P^*x)-C(Rx)}\|^2_{\boldsymbol{L}^{-1}}\\
			\leq& \frac{1}{2}\|(\mathbf{P^*-Q^*})(\mathbf{x}-\mathbf{\bar{x}})\|^2_{\boldsymbol{L}}+ \frac{1}{2}\|\mathbf{P^*x-Rx}\|^2_{\boldsymbol{L}}\\
			=&\frac{1}{2}\langle \mathbf{[(P-Q){\boldsymbol{L}}(P^*-Q^*)+(P-R^{*})\boldsymbol{L}(P^*-R)](x-\bar{x}),x-\bar{x}}\rangle.
			\endaligned
		\end{equation}
		Substituting (\ref{101}), (\ref{102}), (\ref{103}) and (\ref{104}) into \eqref{chang}, followed by multiplying by $2\gamma$, gives (\ref{1.2}). In particular, if Assumption \ref{asm11}(c) holds, then the inner-product on the RHS of \eqref{1.2} is negative and hence $T$ is $\delta$-strongly quasi-nonexpansive whenever $\gamma\in(0,1)$.
	\end{proof}
	The following theorem is our main result regarding the convergence of Algorithm \ref{AlgG}.

	\begin{theorem}\label{theo}
		{
			Suppose that Assumption \ref{asm11}  holds and let $\gamma_k\in[0,1]$ such that $\lim\inf_{k\rightarrow\infty}\gamma_k(1-\gamma_k)>0$. Let the sequences $\{\mathbf{z}^k\} $ and $\{\mathbf{x}^k\}$ be generated by Algorithm \ref{AlgG}. Then the following assertions hold:}
		\begin{itemize}
			\item[{\rm(i)}]   We have $(\Id-T)(\mathbf{z}^k)\to 0$  and $\|\frac{1}{k+1}\sum_{s=0}^{k}(\Id-T)(z^s)\|=O(\frac{1}{\sqrt{k}})$ as $k\to \infty$. It holds $\sum_{i=1}^{n}t_ix_i^k\to 0$ for all $(t_1,\ldots,t_n)\in \mathbb{R}^n$  with $\sum_{i=1}^{n}t_i=0$ as $k\to \infty$.
			\item[{\rm(ii)}] The sequence $\{\mathbf{z}^k\}$ converges weakly to a point $\bar{\mathbf{z}}\in\Fix T$.
			\item[{\rm(iii)}]  The sequence $\{\mathbf{x}^k\}$ converges weakly to $(\bar{x},\dots,\bar{x}) \in \mathcal{H}^n$, where
			$\bar{x}=J_{{d_1}A_1}({d_1}\mathbf{M}_1\mathbf{\bar{z}})$ $\in\zer\left( \sum_{i=1}^{n}A_i + \sum_{i=1}^{m}B_i+ \sum_{i=l}^{l}C_i \right)$ and $d_1M_1$ denotes the first row of the matrix $DM$.
		\end{itemize}					
	\end{theorem}
	\begin{proof}
		{\rm(i)} Since $\zer\left(\sum_{i=1}^{n}A_i+\sum_{i=1}^{m}B_i+\sum_{i=1}^{l}C_i\right)\ne \emptyset$, Lemma
		\ref{lem11} implies $\Fix T\ne \emptyset$.
		Let $\gamma\in(0,1)$ in \eqref{T}, then	Lemma \ref{lem22} implies that $T$ is $\delta$-strongly quasi-nonexpansive for $\delta=\frac{1-\gamma}{\gamma}$. It is easy to verify that $\frac{\gamma_k}\gamma\in[0,1+\delta]$ such that  $\lim\inf_{k\rightarrow\infty}\frac{\gamma_k}\gamma(1-\frac{\gamma_k}\gamma+\delta)>0$. Then using \eqref{Td1} and Lemma \ref{per1}, we obtain that $\{z^k\}$
		is Fej\'{e}r monotone with respect to $\Fix T$  as well as $(\Id-T)(z^{k})\to0$ and $\|\frac{1}{k+1}\sum_{s=0}^{k}(\Id-T)(z^s)\|=O(\frac{1}{\sqrt{k}})$  as $k\to\infty$. Furthermore, it is also concluded that $\{z^k\}$ is bounded.
		
		Next, let $t\in\{t_i\in\mathbb{R}^n:\sum_{i=1}^{n}t_i=0\}=(\ker(M^\top))^\perp=\range M$. Then, there exists $v\in\mathbb{R}^{n-1}$ such that $t=-Mv$ and hence
		\begin{equation*}
			\begin{aligned}
				\sum_{i=1}^{n}t_ix_i^k&=\left( t^\top\otimes \Id\right)\mathbf{x}^k=-((v^{\top}M^{\top})\otimes\Id)\mathbf{x}^k\\
				&=-\left( v^\top\otimes \Id\right)\mathbf{M^*x}^k=\frac{1}{\gamma}\left( v^\top\otimes \Id\right)((\Id-T)(\mathbf{z}^k))\to 0,\quad \hbox{as}\,\, k\to\infty.
			\end{aligned}
		\end{equation*}
		
		{(ii) and \rm(iii)}  Let $\mathbf{x}^k=J_{ \mathbf{DA}}(\mathbf{y}^k)$ where $\mathbf{y}^k=\mathbf{DNx}^k+\mathbf{DMz}^k-\mathbf{DHB(Gx}^k)-\mathbf{D(P-Q)C(Rx}^k)-\mathbf{DQC(P^*x}^k).$ We claim that the sequence $\{\mathbf{x}^k\}$ is bounded. To see this, it follows from Assumption \ref{asm11}(b) and (2) of Remark \ref{rem11}  that
		\begin{equation*}
			x_1^k=J_{d_1A_1}(y_1^k)=J_{d_1A_1}(d_1\mathbf{M}_1\mathbf{z}^k).
		\end{equation*}
		By the nonexpansivity of resolvents and boundedness of $\{\mathbf{z}^k\}$, we get that $\{x^k_1\}$ is  bounded. Using (i), we have that $\{\mathbf{x}^k\}$ is bounded, as claimed.
		Let $ \mathbf{\bar{z}}=(\bar{z}_1,\dots,\bar{z}_{n-1})\in\mathcal{H}^{n-1}$ be an arbitrary weak cluster point of $\{\mathbf{z}^k\}$. Then there exists a point $\mathbf{\bar{x}}\in\mathcal{H}^n$ such that
		$(\mathbf{\bar{z}},\mathbf{\bar{x}})$ is a weak cluster point of $\{(\mathbf{z}^k,\mathbf{x}^k)\}$, where $\mathbf{\bar{x}}=(\bar{x},\dots,\bar{x})$ according to (i).
		Denote $\mathbf{w}^k=\mathbf{DM}\mathbf{z}^k+\mathbf{DN}\mathbf{x}^k$. Then $\mathbf{y}^k=\mathbf{w}^k-\mathbf{DHB(Gx}^k)-\mathbf{D(P-Q)C(Rx}^k)-\mathbf{DQC(P^*x}^k)$. We have that $\mathbf{w}=\mathbf{DM}\mathbf{\bar{z}}+\mathbf{DN\bar{x}}$ is a weak cluster point of $\{\mathbf{w}^k\}$ and let $\mathbf{w}=(w_1,\ldots,w_{n})$. Using  $\range M=\{t_i\in\mathbb{R}^n:\sum_{i=1}^{n}t_i=0\}$,  we deduce
		$\sum_{i=1}^{n}\frac{w_i^k}{d_i}=\sum_{i,j=1}^{n}N_{ij}x_j^k$ from $\mathbf{D}^{-1}\mathbf{w}^k=\mathbf{M}\mathbf{z}^k+\mathbf{N}\mathbf{x}^k$. Define
		the operator $\Phi: \mathcal{H}^n\rightrightarrows \mathcal{H}^n$ is defined by
		\begin{equation}\label{phi}
			\begin{aligned}
				\Phi= &\begin{pmatrix}
					A_1^{-1}\\
					(A_2+{\textstyle \sum_{j=1}^{m}}H_{2j} B_j+\sum_{j=1}^{l}Q_{2j}C_j+\sum_{j=1}^{l}(P-Q)_{2j}C_j)^{-1}\\
					\vdots \\
					(A_{n-1}+{\textstyle \sum_{j=1}^{m}}H_{n-1,j} B_{j}+\sum_{j=1}^{l}Q_{n-1,j}C_j+\sum_{j=1}^{l}(P-Q)_{n-1,j}C_j)^{-1}\\
					A_n+{\textstyle \sum_{j=1}^{m}}H_{nj} B_{j}+\sum_{j=1}^{l}Q_{nj}C_j+\sum_{j=1}^{l}(P-Q)_{nj}C_j
				\end{pmatrix}\\
				&+\begin{pmatrix}
					0 & 0 & 0  &\cdots  &0  & -\Id\\
					0 & 0 & 0 &\cdots  &0  & -\Id \\
					\vdots&\vdots   & \vdots  &\ddots &\vdots   &\vdots  \\
					0 & 0 & 0  &\cdots  &0  & -\Id \\
					\Id& \Id &\Id&\cdots & \Id &0
				\end{pmatrix}.
			\end{aligned}
		\end{equation}  Then, from \eqref{main-scheme}, we get
		\begin{equation}\label{FT}
			\aligned
			\Phi&\begin{pmatrix}\frac{1}{d_1}( w_1^k-x_1^k)\\\frac{1}{d_2}(w_2^k-x_2^k)+a^k_2+ b_2^k \\ \vdots \\\frac{1}{d_{n-1}}( w_{n-1}^k-x_{n-1}^k)+ a^{k}_{n-1}+b_{n-1}^k\\x_n^k\end{pmatrix}
			\ni \begin{pmatrix}x_1^k-x_n^k \\ x_2^k-x_n^k\\ \vdots \\ x_{n-1}^k-x_n^k\\ {\textstyle \sum_{i=1}^{n}}\frac{1}{d_i} (w_i^k-x_i^k) + {\textstyle \sum_{i=2}^{n}}(a^k_i+b_i^k) \end{pmatrix}\\
			&\qquad\qquad\qquad\qquad=\begin{pmatrix} x_1^k-x_n^k \\ x_2^k-x_n^k\\ \vdots \\ x_{n-1}^k-x_n^k\\ {\textstyle \sum_{i,j=1}^{n}}N_{ij}x_j^k-{\textstyle \sum_{i=1}^{n}}\frac{1}{d_i} x_i^k +  {\textstyle \sum_{i=2}^{n}}(a^k_i+b_i^k)\end{pmatrix},
			\endaligned
		\end{equation}
		where $a_i^k={\textstyle \sum_{j=1}^{m}}H_{ij} B_{j}(x_i^k)-{\textstyle \sum_{j=1}^{m}}H_{ij} B_{j}( {\textstyle \sum_{h=1}^{i-1}}G_{jh}x_h^k) $  and
		$b_i^k=\sum_{j=1}^{l}Q_{ij}C_j(x_i^k)+\sum_{j=1}^{l}(P-Q)_{ij}C_j(x_i^k)-\sum_{j=1}^{l}Q_{ij}C_j(\sum_{h=1}^{i-1}P_{hj}x_h^k)-\sum_{j=1}^{l}(P-Q)_{ij}C_j(\sum_{h=1}^{i-1}R_{jh}x_h^k)$ for $i=2,\ldots,n$.
		According to (i) and Lipschitz continuity  of $B_i$ (for $i=1,\dots,m$) and $C_j$ (for $j=1,\dots,l$),  we have $\lim_{k \to \infty} a_i^k=0$ and $\lim_{k \to \infty} b_i^k=0$ for $i=2,\ldots,n.$	As the sum of two maximally monotone operators is again maximally monotone provided that one of the operators has full domain \cite[Corollary 24.4 (i)]{BC2011}, it follows that $\Phi$ is maximally monotone. Consequently, its graph is sequentially closed in the weak-strong topology  \cite[Proposition 20.32]{BC2011}.  Note also that the RHS of (\ref{FT}) converges strongly to zero as a consequence of (i) and (a) of Remark \ref{rem11}(2). Taking the limit  along a subsequence of $\{(\mathbf{z}^k,\mathbf{x}^k)\}$ which converges weakly to $(\mathbf{\bar{z}},\mathbf{\bar{x}})$ in (\ref{FT}), and using the weak-strong topology of $\Phi$, we obtain
		$$\Phi
		\begin{pmatrix}
			\frac{1}{d_1}(w_1-\bar{x})\\
			\vdots \\
			\frac{1}{d_{n-1}}(w_{n-1}-\bar{x})\\
			\bar{x}
		\end{pmatrix}\ni \begin{pmatrix}
			0 \\
			\vdots \\
			0 \\
			0
		\end{pmatrix}.$$ Then,  from the definition of $\Phi$ in \eqref{phi}, it follows
		\begin{equation*}\label{phi0}
			\left\{\begin{array}{l}
				A_{1}(\bar{x}) \ni \frac{1}{d_1}(w_1-\bar{x} ),\\
				(A_{i}+{\textstyle \sum_{j=1}^{m}}H_{ij} B_{j}+\sum_{j=1}^{l}Q_{ij}C_j+\sum_{j=1}^{l}(P-Q)_{ij}C_j)(\bar{x})\ni \frac{1}{d_i}(w_i-\bar{x}),\\ ~~~~~~~~~~~~~~~~~~~~~~~~~~~~~~~~~~~~~~~~~~~~~~~~~~~~~~~~~~~~~~~~~~~~~~~~~~~~\forall i \in[2, n-1],\\
				(A_n+{\textstyle \sum_{j=1}^{m}}H_{nj} B_{j}+\sum_{j=1}^{l}Q_{nj}C_j+\sum_{j=1}^{l}(P-Q)_{nj}C_j)(\bar{x}) \ni-\sum_{i=1}^{n-1}\frac{1}{d_i}(w_{i}-\bar{x}),
			\end{array}\right.
		\end{equation*}
		which implies $(\mathbf{\bar z},x)\in\Omega$.
		It follows  $\mathbf{\bar{z}}\in\Fix T$ and $\bar{x}\in\zer\big ( \sum_{i=1}^{n}A_i+\sum_{i=1}^{m}B_i+ \sum_{i=1}^{l}C_i  \big )$ from Lemma \ref{lem11}(c). Hence $\omega_w(\mathbf{{z}}^k)\subseteq\Fix(T)$. Using Lemma \ref{lle}, we get that $\{\mathbf{z}^k\}$ converges  weakly to a point in $\Fix T$. Since	  $\bar{x}= J_{{d_1}A_1}(w_1)=J_{{d_1}A_1}({d_1}\mathbf{M}_1\mathbf{\bar{z}})$, $\mathbf{\bar{x}}=(\bar{x},\dots,\bar{x})$  is the unique weak sequential cluster point of the  sequence $\{\mathbf{x}^k\}$. Therefore  $\{\mathbf{x}^k\}$ converges  weakly to $\mathbf{\bar{x}}$.	
	\end{proof}

	\section{Numerical experiment}
	\label{Sect5}
	
	In this section, we first conduct tests on the selections of  parameters and present some heuristic strategies and further propose a new algorithm based on these strategies. Then we  compare the proposed algorithm with three methods in \cite{FBWR} through a numerical experiment.
	
	\subsection{Setting and selection of parameters}\label{test_que}
	In this subsection, we present a series of numerical
	experiments to investigate the influence of the matrices $M,U, H,G,P,Q,R$ on the performance of
	Algorithm \ref{AlgG}, and we provide suggestions for their selections to achieve good performance.

	\subsubsection{Problem description}\label{qu}
	For $d,n,m,l \in \mathbb{N}$, $0 \leq \delta_1 \leq \delta_2$, a matrix $\Psi \in \mathbb{R}^{m\times d}$, a vector $y \in \mathbb{R}^m$, a sample of points $\{\xi_i\}_{i=1}^n \subset \mathbb{R}^d$, we consider the following convex optimization problem:
	\begin{equation}\label{Te}
		\min_{x \in \mathbb{R}^d} f(x)= \|x - \xi_1\| + \dots + \|x - \xi_n\| + H_{\delta_1,\delta_2}(\Psi x - y)+\frac{1}{2}x^\top \Theta_{1}x+ \dots+\frac{1}{2}x^\top \Theta_{l}x,
	\end{equation}
	where 	 $\{\Theta_i\}_{i=1}^l\subset \mathbb{R}^{d\times d}$ is a randomly generated  matrix,  which is symmetric but not negative definite, and $H_{\delta_1,\delta_2} : \mathbb{R}^m \to \mathbb{R}$ is a Huber-like smooth function defined for all $z := (z_1,\dots,z_m) \in \mathbb{R}^m$ by:
	\[
	H_{\delta_1,\delta_2}(z) := \sum_{i=1}^m h_{\delta_1,\delta_2}(z_i), \quad
	h_{\delta_1,\delta_2}(z_i) :=
	\begin{cases}
		0 & \text{if } |z_i| \leq \delta_1, \\
		\frac12 (z_i - \delta_1)^2 & \text{if } |z_i| \in [\delta_1, \delta_2], \\
		(\delta_2 - \delta_1)|z_i| - \frac12 (\delta_2^2 - \delta_1^2) & \text{else}.
	\end{cases}
	\]
	The first-order optimality condition implies that the problem \eqref{Te} is equivalently expressed as the inclusion problem \eqref{ABC}. To this end,		we take
	$
	B_i(x) := h'_{\delta_1,\delta_2}(\Psi_i x - y_i) \Psi_i^\top
	$
	for all $i \in \{1,\ldots,m\}$,
	where \(\Psi_i\) is the \(i\)-th row of \(\Psi\). Since	the function \(h_{\delta_1,\delta_2}\) is differentiable with a 1-Lipschitz continuous gradient, we have \(\sigma_i = \frac{1}{\|\Psi_i\|_2^2}\), which can be computed directly.
	We  take  $C_i(x):=\Theta_{i}x$ with $L_i=\|\Theta_{i}\|_2$, $i \in \{1,\ldots,l\}$. Regarding the nonsmooth terms, we consider \(A_i := \partial g_i\) with \(g_i(x) := \|x - \xi_i\|\) for all $i \in \{1,\ldots,n\}$. Note that for all \(\tau > 0\), \(J_{\tau A_i}\) coincides with \(\mathrm{prox}_{\tau g_i}\), which admits a simple closed-form expression via a standard soft-thresholding operation.

	\subsubsection{Algorithm design}\label{can}
	To test the influence of the  matrices, we generate instances of $M,H,G,P,Q,R,K$,  which meet Assumption \ref{asm11}. Then  we present three heuristics based on the numerical performances.
	For ease of later use, we denote by $\mathcal{U}(I)$ the uniform distribution on the interval $I \subset\mathbb{R}$.
	
	We take $n=20,m=15,l=10,d=2,\delta_1=1,\delta_2=2$ and sample  $\Psi$ from  $\mathcal{U}(-2.5,2.5)^{20\times2}$.  The elements of $y$ and $\xi_i$, $i=1,\ldots,n$ are randomly sampled from $\mathcal{U}(0,1)$ and $\mathcal{N}(0,5^2)$, respectively. We  take $\Theta_i=\overline{\Theta}_{i}^T\overline{\Theta}_{i}$, $i=1,\ldots,l$, where $\overline{\Theta}_{i}\in\mathcal{U}(0,1)^{2\times2}$.  In all experiments, the initial values are set to zero.
	\vskip 2mm
	
	\noindent
	\textbf{(i) The random choice of $H,G,P,Q,R$}
	\vskip 2mm
	
	\noindent
	To generate general $H,~G^\top\in\mathbb{R}^{n\times m},$ and $P,~Q,~R^\top\in\mathbb{R}^{n\times l},$  we first sample $\widetilde H,$  $\widetilde G^\top$, $\widetilde P,~\widetilde Q,~\widetilde R$ from $\mathcal{U}(I_H)^{n\times m},~\mathcal{U}(I_{G})^{n\times m}$, $\mathcal{U}(I_P)^{n\times l},~\mathcal{U}(I_Q)^{n\times l},~\mathcal{U}(I_{R})^{n\times l}$. Then we  generate randomly a nondecreasing vector $\overline E=(\overline E_1,\ldots,\overline E_n)\in \{0,...,m\}^n$ and a  pair of nondecreasing vectors $\widetilde E=(\widetilde E_1,\ldots,\widetilde E_n)\in \{0,...,m\}^n$, $F=(F_1,\ldots,F_n)\in \{0,...,l\}^n$ satisfying $\widetilde E_{i-1}\geq F_i$. We set
	\begin{equation*}
		\begin{aligned}
			&\widetilde H_{ij} = 0, \widetilde G_{hi} = 0,\text{for all}~h\leq \overline E_i<j,~\widetilde P_{ij} = 0,\text{for all}~\widetilde E_i<j \leq F_i,\\
			&\widetilde Q_{ij} = 0,\text{for all}~j > F_i,~\widetilde R_{ji} = 0,\text{for all}~j \leq \widetilde E_i.\\
		\end{aligned}
	\end{equation*}
	Eventually, we normalize the sum to one as:
	{\small
		\begin{equation*}
			\begin{aligned}
				H_{ij}=\frac{ \widetilde H_{ij}}{\sum_{k=1}^{n}\widetilde H_{kj}},~G_{ij}=\frac{\widetilde G_{ij}}{\sum_{k=1}^{m}\widetilde G_{ik}},~P_{ij}=\frac{\widetilde P_{ij}}{\sum_{k=1}^{n}\widetilde P_{kj}},~Q_{ij}=\frac{\widetilde Q_{ij}}{\sum_{k=1}^{n}\widetilde Q_{kj}},~R_{ij}=\frac{\widetilde R_{ij}}{\sum_{k=1}^{l}\widetilde R_{ik}}.
			\end{aligned}
		\end{equation*}
	}
	The above procedure yields a causal pair of  matrices $H,G$, and a relatively causal triple of $P,Q,R$, which satisfy Assumption \ref{asm11}(b).
	
	\vskip 2mm
	\noindent
	\textbf{(ii) The choices of $M$ and $\mathcal{L}$}
	\vskip 2mm
	
	\noindent
	We pick the matrix $\mathcal{L}$ in \eqref{Simp} through  three distinct ways:
	\begin{itemize}
		\item[(a)]
		$\mathcal{L}=MM^\top$, where $M$ is defined by the following:.
		\begin{equation*}\label{samp}
			M=(I_n-\frac{1}{n}e_ne_n^\top)\widetilde M,~\text{for}~\widetilde M~\text{sampled from}~ \mathcal{U}(I_M)^{n\times (n-1)}.
		\end{equation*}
		\item[(b)] The graph Laplacian of a connected Watts--Strogatz small-world graph, where the degree of each node is a  integer randomly chosen from $\{1,\ldots,n\}$.
		\item[(c)]  $\mathcal{L}=nI_n-e_ne_n^\top$, which is called complete graph Laplacian.
	\end{itemize}
	For each of the above cases, we normalize $\mathcal{L}$ such that $\|\mathcal{L}\|_2=1$. It is easy to see that for the case (a), Assumption \ref{asm11}(a) necessarily holds. For the cases (b) and (c), by \cite[Proposition 2.16]{fran}, there exists the factorization $\mathcal{L}=MM^\top$ such that $M$ satisfies Assumption \ref{asm11}(b).
	\vskip 1mm
	
	We numerically investigate how the choice of $\mathcal{L}$
	affects the overall performance. To avoid finding the corresponding $M$, we implement the equivalent version (\ref{Simp}) of Algorithm \ref{AlgG}.
	We set $U=0$, generate matrices $H,G,P,Q,R$ randomly following the procedure in (i), and execute (\ref{Simp}) over 200 independent runs, each consisting of 100 iterations.
	We track the final value of the objective function residual and the associated algebraic connectivity, defined as the smallest non-zero eigenvalue of $\mathcal{L}$ and  denoted by $\varrho$.  The results are visualized in Figure \ref{fig1}(a), where the blue dots correspond to $\mathcal{L}$ in (a),  the black ones to $\mathcal{L}$ in (b), and the red one to the complete graph Laplacian. This figure demonstrates that within the reflected forward-backward framework of this study, using  the complete graph Laplacian yields superior results. It also indicates that graph Laplacians mostly deliver better performance than matrices randomly sampled according to  (a).
	
	\begin{figure}[H]
		\centering
		\begin{subfigure}[T]{0.4\columnwidth}
			\small{(a)}\includegraphics[width=\textwidth]{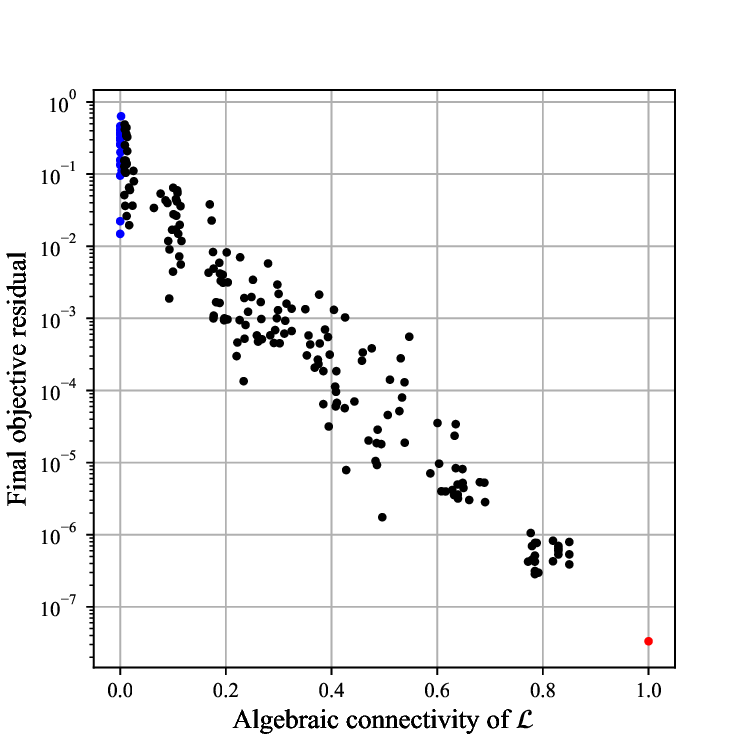}
		\end{subfigure}
		\begin{subfigure}[T]{0.4\columnwidth}
			\small{(b)}\includegraphics[width=\textwidth]{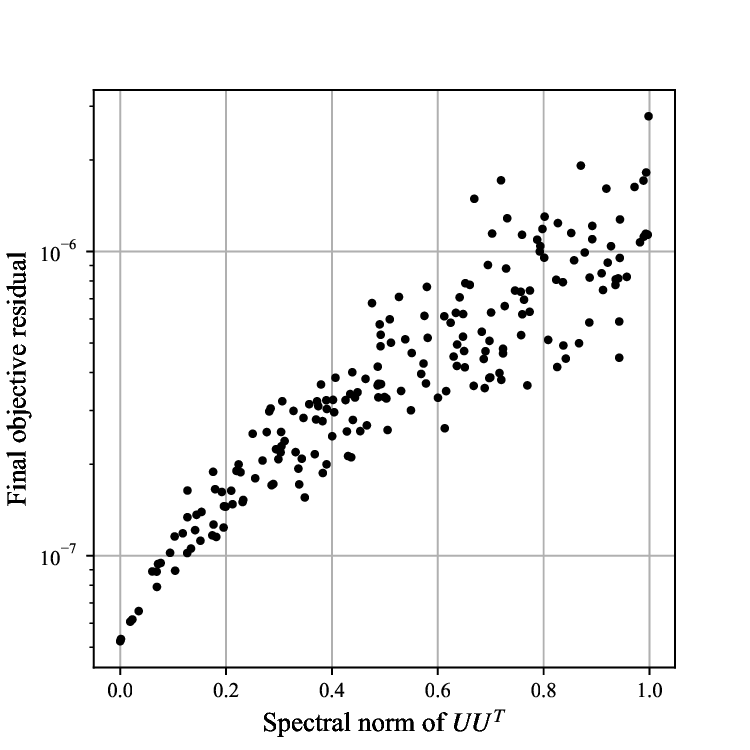}
		\end{subfigure}
		\caption{(a) Testing the influence of the spectrum
			of $\mathcal{L}$: final residual against $\varrho$. (b) Testing the influence of $U$: final residual against $\|UU^\top\|_2$.}
		\label{fig1}
	\end{figure}
	
	\noindent
	\textbf{(iii) The random choice of $U$}
	\vskip 2mm
	
	\noindent
	We investigate how the choice of $U$ affects the overall performance.
	We use the same experimental setup as in (ii), but here we keep a fixed (randomly generated) Laplacian 	$\mathcal{L}$. We randomly generate a matrix $U$ using the procedure from the uniform distribution as before. We compute $UU^\top$ and rescale it so that its spectral norm $\|UU ^\top\|_2$ matches a value randomly selected from the interval [0,1]. We record the final objective function residual and the corresponding value of $\|UU ^\top\|_2$. The results are illustrated in Figure \ref{fig1}(b), verifying that   $U=0$ is an optimal choice, which coincides with Remark \ref{rem11}(1).
	
	\vskip 2mm
	\noindent
	\textbf{(iv) The choice of $H,G,P,Q,R$ through a constrained convex programming}
	\vskip 2mm
	
	\noindent	
	We first demonstrate that using individual estimates of the cocoercivity constants of $B_i$ and  Lipschitz constants of $C_i$ can improve the algorithm's performance. To this end,  we first sample $\Psi$ and $\Theta_j$ uniformly. Then we use the same matrix but separately scale a subset of randomly selected two rows of $\Psi$ and $\Theta_j$ by a factor of 5 to introduce heterogeneity into the data fidelity term.
	
	We test two setups: (I) using potentially distinct $\sigma_i$ and $L_i$, and (II) using a uniform $\sigma_{\min} = \min\{\sigma_1, \dots, \sigma_m\}$ and  $L_{\max} = \max\{L_1, \dots, L_l\}$.  We generate matrices $H,G$ and $P,Q,R$ randomly following the procedure in (i) and  run Algorithm \ref{AlgG} for 20 trials.
	We measure the objective function residual with the iterations  in Figure \ref{fig3}(a), which  clearly show that accounting for the heterogeneity of the data significantly boosts the method's performance.
	
	Next we numerically investigate how the choices of $H,G,P,Q,R$ affects the overall performance. 	To this end, we let
	\begin{equation*}
		W = \frac{1}{2}(H-G^{\top}){\Sigma^{-1}}(H^{\top}-G)+(P-Q) L(P^\top-Q^\top)+(P-R^{\top}) L(P^\top-R).
	\end{equation*}
	We  analyze the performance of Algorithm \ref{AlgG} in relation to the spectral norm of $W$.
	
	We sample $H,G,P,Q,R$ randomly 500 times following the procedure in (i) and run Algorithm \ref{AlgG}. During each run, we track the objective function residual across iterations and record $\|W\|_2$. The results are displayed in Figure \ref{fig3}(b), where each line  represents a specific trial and the color indicates the magnitude of $\|W\|_2$.
	It is observed that the best performance is achieved when $H,G, P,Q,R$ are chosen to minimize the norm of $\|W\|_2$.
	
	\begin{figure}[h!]
		\centering
		\begin{subfigure}[t]{0.4\columnwidth}
			\adjustbox{valign=t}{\small{(a)}\includegraphics[width=\textwidth]{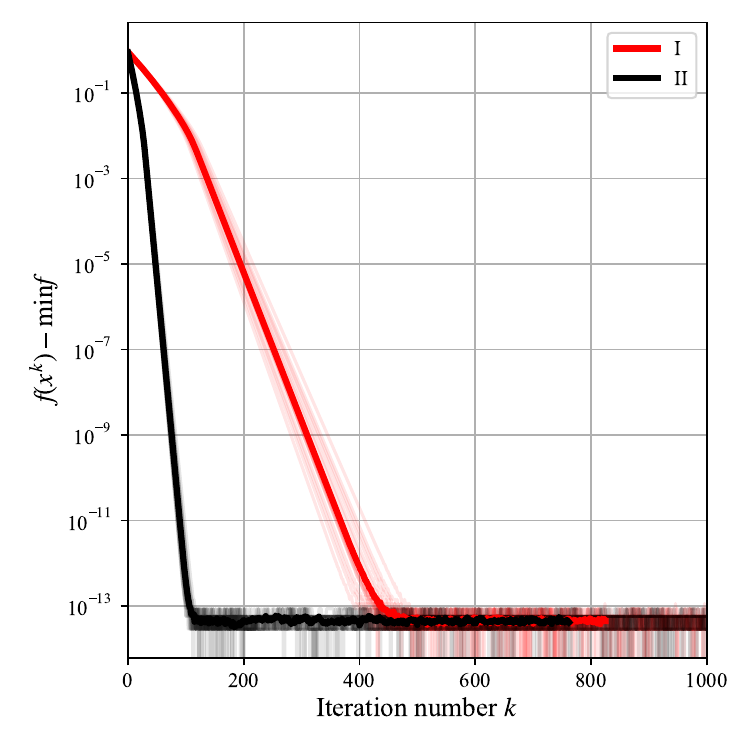}}
		\end{subfigure}
		\begin{subfigure}[t]{0.4\columnwidth}
			\adjustbox{valign=t}{\small{(b)}\includegraphics[width=\textwidth]{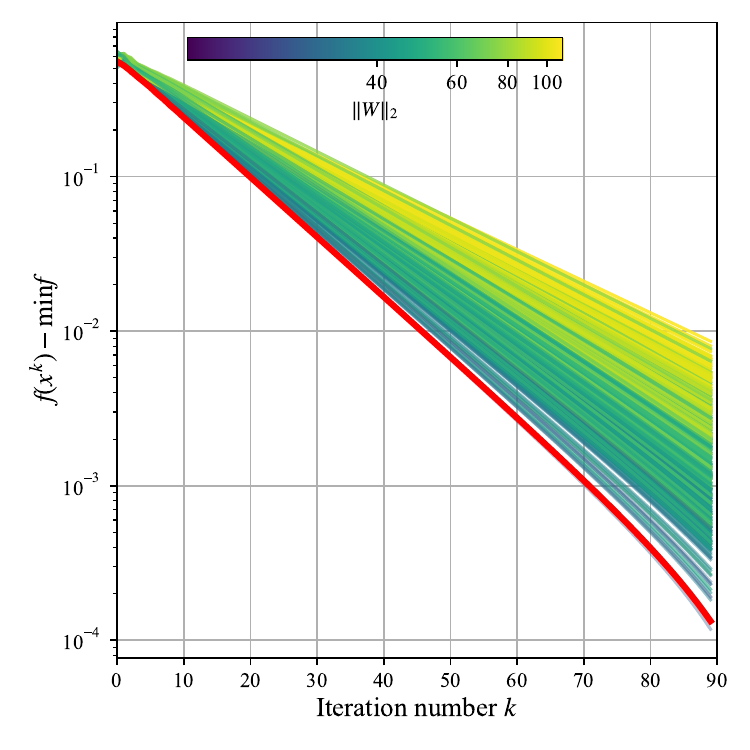}}
		\end{subfigure}
		\caption{(a) Accounting for heterogeneity of data enhances performances. (b) Objective value decrease high lighting $\|W\|_2$.}
		\label{fig3}
	\end{figure}

	The above observation suggests selecting the causal pair of matrices $H,G$ and  relatively causal triple of matrices  $P,Q,R$ by minimizing $\|W\|_2.$ Let
	$W_1=\frac{1}{\sqrt{2}}(H-G^\top)\sqrt{\text{diag}(\Sigma^{-1})},$
	$W_2=(P-Q)\sqrt{\text{diag}(L)}$, $W_3=(P-R^\top)\sqrt{\text{diag}(L)}$ and $\Upsilon=[W_1,W_2,W_3]$, then
	$$
	W =W_1 W_1^\top+W_2 W_2^\top+W_3 W_3^\top=[W_1,W_2,W_3][W_1,W_2,W_3]^\top=\Upsilon\Upsilon^\top,
	$$
	which  is obviously symmetric. By the symmetry of $W$ and the definition of the $2$-norm of a matrix, we have $\|W\|_2=\|\Upsilon\|_2^2$. Therefore,  we can minimize $\|W\|_2$ by minimizing $\|\Upsilon\|_2$, i.e.,
	solving the following convex optimization problem:
	\begin{equation}\label{qu2}
		\begin{aligned}
			&\underset{\substack{H,G^\top\in\mathbb{R}^{n\times m}\\P,Q,R^\top\in\mathbb{R}^{n\times l}}}{\min}\|[\frac{1}{\sqrt{2}}(H-G^\top)\sqrt{\text{diag}(\Sigma^{-1})},(P-Q)
			\sqrt{\text{diag}(L)},(P-R^\top)\sqrt{\text{diag}(L)}]\|_2\\
			&\text{subject to~~}  H^\top e_n =Ge_n=e_{m},~P^\top e_n =Q^\top e_n=e_{l},~R e_n=e_{l}, \\
			&~~~~~~~~~~~~~~~H\in \mathcal{S}(\overline E),~G^\top\in \mathcal{S}^c(\overline E),~Q\in \mathcal{S}(F),P\in \mathcal{S}^c(F)\cap\mathcal{S}(\widetilde E),~R\in \mathcal{S}^c(\widetilde E).
		\end{aligned}
	\end{equation}	
	
	Based on the above tests, we obtain the following heuristic strategies for the matrix selections.
	\begin{itemize}
		\item[(a)]Following (ii), we choose
		$M\in\mathbb{R}^{n\times(n-1)}$ from the factorization $\mathcal{L}=MM^{T}$, where $\mathcal{L}=nI_n-e_ne_n^\top$.
		\item[(b)]Following (iii), we choose $U=0$.
		\item[(c)]Following (iv), we choose matrices $H$, $G,P,Q,R$ by solving the minimization problem \eqref{qu2}.
	\end{itemize}
	We refer to Algorithm \ref{AlgG} with  matrices given by the above heuristic strategies  as Complete-Reflected-Forward-Backward (CRFB).
	
	\subsection{Application to a regularized saddle-point problem}
	In this section, we   compare Complete-Reflected-Forward-Backward with the distributed forward-backward method with reflection terms (Distributed-Forward-Backward-Reflection: DFBR), generalized parallel Davis--Yin method with reflected terms (Parallel-Davis--Yin-Reflected: PDYR) and generalized sequential Davis--Yin method with reflected terms (Sequential-Davis--Yin-Reflected: SDYR) in \cite{FBWR}, which  are also given in Appendix.
	
	\subsubsection{Problem description}
	\label{sec4.1}
	
	We study a class of zero-sum matrix games with regularization-like terms. This model can be viewed as a generalization of \cite{NO}, where the scalar parameters preceding the two regularization terms are replaced by matrices. Moreover, our model is also structurally related to the extensive-form games with convex-concave saddle-point structure studied in \cite{AAAI}, where the dilated convex functions incorporated into the payoff can be interpreted as regularization terms.
	
	Specifically, there are two teams, each consisting of $n$ players. The payoff matrix for Team 1 is given by $\Theta = \sum_{i=1}^{l} \Theta_i \in \mathbb{R}^{d_2 \times d_1}$, whereas that of Team 2 is $-\Theta$. Each player in Team 1 is paired with a counterpart from Team 2, and the $i$-th player pair is linked to the payoff matrices $(\Theta_i, -\Theta_i)$. The Nash equilibria of this game coincides with the saddle points of the subsequent min-max problem:\\
	\begin{equation}\label{ball}
		\min_{u\in\Delta^{d_1}}\max_{v\in\Delta^{d_2}}F(u,v)=\sum_{j=1}^l \langle \Theta_j u, v\rangle+\frac{1}{2}\sum_{j=1}^m\|V_iu\|^2 - \frac{1}{2}\sum_{j=1}^m\|J_iv\|^2
	\end{equation}
	where  $\Delta^{d_1}=\{u=(u_1,...,u_{d_1})\in[0,+\infty)^{d_1}\mid\sum_{i=1}^{d_1}u_i=1\}$, $\Delta^{d_2}=\{v=(v_1,...,v_{d_2})\in[0,+\infty)^{d_2}\mid\sum_{i=1}^{d_2}v_i=1\}$ are the unit simplex. The problem (\ref{ball}) can be also expressed as \eqref{saddle} with
	\begin{equation}\label{ball2}
		\min_{u\in\mathbb{R}^{d_1}}\max_{v\in\mathbb{R}^{d_2}} \sum_{j=1}^{n}\left(\iota_{\Delta^{d_1}}(u) -\iota_{\Delta^{d_2}}(v)\right)+\sum_{j=1}^l \langle \Theta_j u, v\rangle+\frac{1}{2}\sum_{j=1}^m\left(\|V_iu\|^2 - \|J_iv\|^2\right),
	\end{equation}
	where $\iota_{\Delta^{d_1}}$ and $\iota_{\Delta^{d_2}}$ are the indicators of $\Delta^{d_1}$ and $\Delta^{d_2}$, respectively.
	By the first-order optimality condition, the problem \eqref{ball2}  can be formulated as an inclusion problem
	\begin{equation*}
		\label{ABC-2}
		\hbox{find}\ x\in\mathbb{R}^{d_1+d_2}\ \hbox{such\ that}\ 0\in \sum_{j=1}^{n}A_j(x)+\sum_{j=1}^{m}B_j(x)+\sum_{j=1}^{l}C_j(x),
	\end{equation*}
	where
	\begin{equation*}
		x = \begin{bmatrix}
			u\\
			v
		\end{bmatrix}\in \mathbb{R}^{d_1}\times \mathbb{R}^{d_2},
		~A_j=
		\begin{bmatrix}
			N_{\Delta^{d_1}}&\mathbf{0}\\
			\mathbf{0}&N_{\Delta^{d_2}}
		\end{bmatrix},
		~	B_j=
		\begin{bmatrix}
			V_j^\top V_j&\mathbf{0}\\
			\mathbf{0}&J_j^\top J_j
		\end{bmatrix},~		C_j=
		\begin{bmatrix}
			\mathbf{0}&\Theta_j\\
			-\Theta_j&\mathbf{0}
		\end{bmatrix},
	\end{equation*}
	and $N_{\Delta^{d_1}}$ and $N_{\Delta^{d_2}}$ are the normal cones of $\Delta^{d_1}$ and $\Delta^{d_1}$, respectively.
	Recall that $A_j$,  $j=1,\dots,n$ is maximally monotone, $B_j$, $j=1,\dots,n-1$ is  $\sigma_j$-cocoercive with $\sigma_j=\frac{1}{\max\left\{ \|V_j\|_2^2,\ \|J_j\|_2^2 \right\}}$, and $C_j$, $j=1,\dots,n-2$ is monotone and $L_j$-Lipschitz continuous constant with $L_j=\|\Theta_j\|_2$. Thus the problem \eqref{ball2} is a special case of the problem \eqref{ABC}.

	\subsubsection{Experiment settings and results}

	In experiments, we take $n=20$, $m=19$, $l=18$, and  let $d_1=d_2=d=30$.	
	We let $V_j=J_j=\Omega_j$, and
	generate $\Omega_j$ and $\Theta_j$
	randomly in the following four different ways with random
	number generator \textit{seed} = 10.

	\begin{itemize}
		\item[{\rm(a)}]The elements in $\Omega_j$ and $\Theta_j$ are independently sampled from uniform distribution $\mathcal{U}(-1,1)$ and $\mathcal{U}(-10,10)$, respectively.
		\item[{\rm(b)}]The elements in $\Omega_j$ and $\Theta_j$ are independently sampled from normal distribution $\mathcal{N}(0,1)$ and $\mathcal{N}(0,10)$, respectively.
		\item[{\rm(c)}]The elements in $\Omega_j$ and $\Theta_j$ are independently sampled from exponential distribution $Exp(0.2)$ and $ Exp(2)$, respectively.
		\item[{\rm(d)}]The elements in $\Omega_j$ and $\Theta_j$ are independently sampled from poisson distribution $Poisson(0.2)$ and $Poisson(2)$, respectively.
	\end{itemize}
	With the above setup already in place,  we  take  take zero initial values for all methods.
	Referring to \cite{AAAI} and \cite{GAP}, we define the primal-dual gap as follows:
	\begin{equation*}
		Gap(u,v) = \text{max}_{\overline{v}\in \Delta^d} F(u,\overline{v})	-\text{min}_{\overline{u}\in \Delta^d} F(\overline{u},v).
	\end{equation*}	
	The convex optimization subproblems $\text{max}_{\overline{v}\in \Delta^d} F(u^k,\overline{v})$ and $\text{min}_{\overline{u}\in \Delta^d} F(\overline{u},v^k)$ in the $Gap(u^k,v^k)$ are solved by using CVX with  MOSEK solver.
	
	Figure \ref{fig2} depicts the decay of the primal-dual gap $Gap(u^k,v^k)$ with the number of iterations, which illustrates that Complete-Reflected-Forward-Backward decreases the fastest among the four algorithms.	We report the number of iterations  required to achieve $Gap(u^k,v^k)<\epsilon$ for $\epsilon\in\{10^{-3},10^{-5},10^{-7}\}$ in  Table \ref{comp}. It can be observed that Complete-Reflected-Forward-Backward performs the best in all cases and for all error tolerances and the behavior of Distributed-Forward-Backward-Reflection is the worst.
	
	\begin{figure}[H]
		\centering
		\begin{subfigure}[T]{0.49\columnwidth}
			\includegraphics[width=\textwidth]{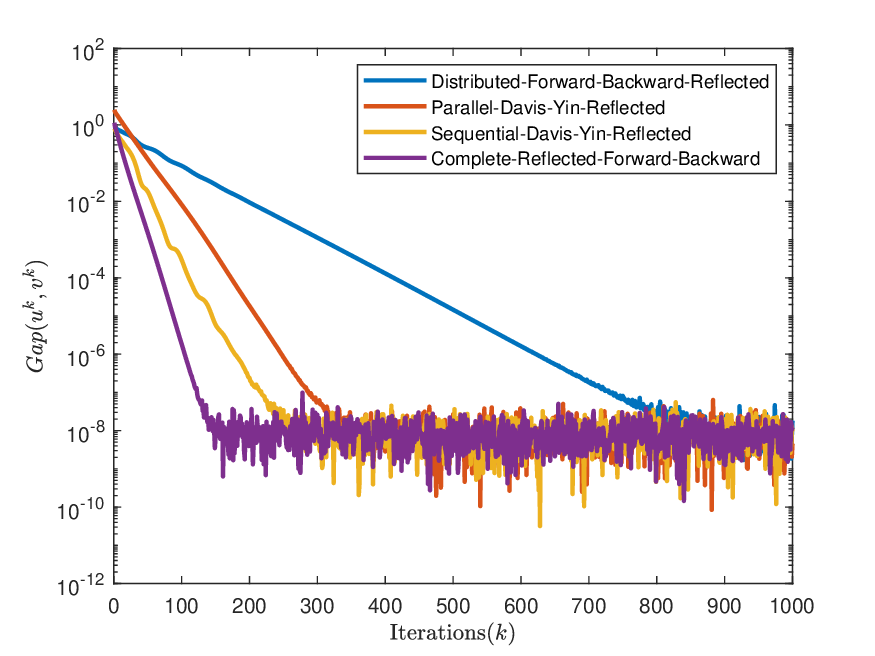}
			\caption*{(a)}
		\end{subfigure}
		\begin{subfigure}[T]{0.49\columnwidth}
			\includegraphics[width=\textwidth]{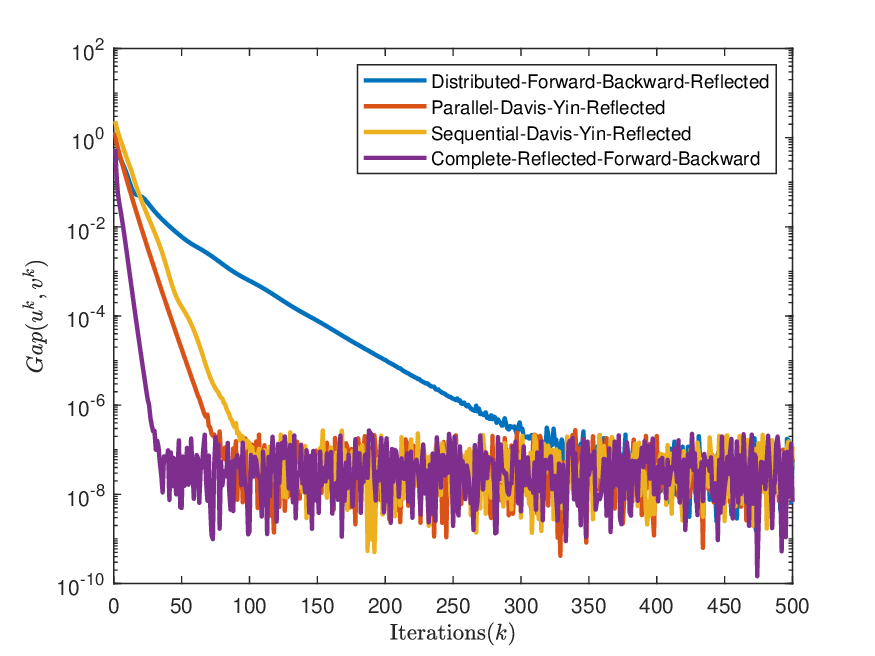}
			\caption*{(b)}
		\end{subfigure}
		\begin{subfigure}[T]{0.49\columnwidth}
			\includegraphics[width=\textwidth]{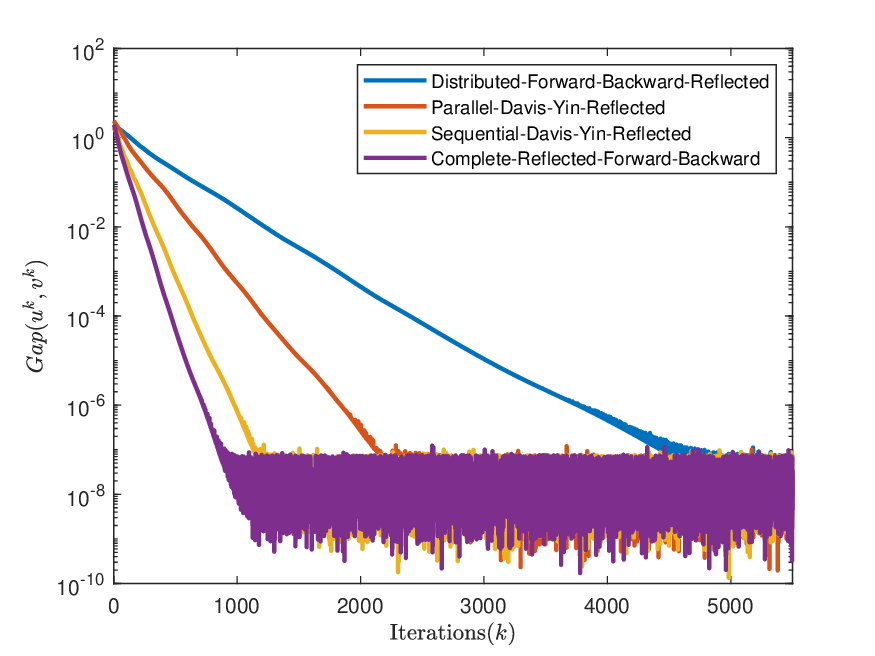}
			\caption*{(c)}
		\end{subfigure}
		\begin{subfigure}[T]{0.49\columnwidth}
			\includegraphics[width=\textwidth]{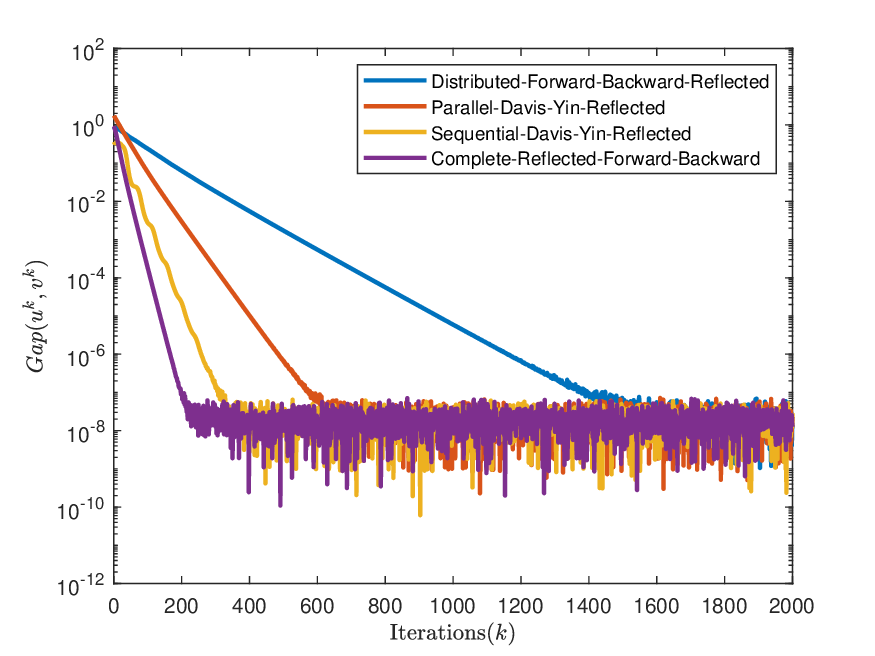}
			\caption*{(d)}
		\end{subfigure}
		\caption{Decay of the primal-dual gap with with the number of iterations for four methods under four different matrices setting.}
		\label{fig2}
	\end{figure}
	
	\begin{table}[htbp]
		\centering
		\caption{Numerical comparisons  of four methods under four different matrices setting.}
		\setlength{\tabcolsep}{18pt}
		\begin{tabular}{cccccc}
			\toprule
			Matrix & $\epsilon$ &\makecell[l]{\footnotesize CRFB} &
			\makecell[l]{\footnotesize DFBR} &
			\makecell[l]{\footnotesize PDYR} &
			\makecell[l]{\footnotesize SDYR} \\
			\cmidrule{1-6}
			\multirow{3}{*}{(a)}
			& $10^{-3}$   & 54  & 306 & 137 & 79 \\
			& $10^{-5}$   &  88 & 518 & 210 & 145\\
			& $10^{-7}$   &  123 & 730 &287 & 218 \\
			\cmidrule{1-6}
			\multirow{3}{*}{(b)}
			& $10^{-3}$   & 12 & 88 &  31 & 40 \\
			& $10^{-5}$   & 21 & 201 & 54 & 67 \\
			& $10^{-7}$   & 31 & 318 & 79 & 101  \\
			\cmidrule{1-6}
			\multirow{3}{*}{(c)}
			& $10^{-3}$   & 352 & 1813 &  920 & 482 \\
			& $10^{-5}$   & 584 & 3023 & 1524 & 802 \\
			& $10^{-7}$   & 847 & 4441 & 2122 & 1132 \\
			\cmidrule{1-6}
			\multirow{3}{*}{(d)}
			& $10^{-3}$   & 79 & 547 &  238 & 124 \\
			& $10^{-5}$   & 137 & 954 & 402 & 213  \\
			& $10^{-7}$   & 198 & 1379 & 574 & 308 \\
			\bottomrule
		\end{tabular}
		\label{comp}
	\end{table}

	\FloatBarrier
	\bibliographystyle{plain}

\begin{thebibliography}{00}	
		
		
		\bibitem{Anton}
		\AA kerman, A.,  Chenchene, E., Giselsson, P. et al. Splitting the forward-backward algorithm: a full characterization. 	https://arxiv.org/abs/2504.10999
		
		\bibitem{Alac2021}
		Alacaoglu, A., Malitsky, Y., Cevher, V. Forward-reflected-backward method with variance reduction.
		Computational Optimization and Applications. 80, 321--346 (2021).
		
		\bibitem{ring}
		Arag\'{o}n--Artacho, F.J., Malitsky, Y., Tam, M.K. et al. Distributed forward-backward methods for ring networks. Computational Optimization and Applications. 86, 845--870 (2023).
		
		\bibitem{fran}
		Arag\'{o}n--Artacho, F.J.,  Campoy, R., L\'{o}pez-Pastor, C. Forward-backward algorithms devised by graphs. SIAM Journal on Optimization 35(4),  2423--2451 (2025).
		
		\bibitem{BC2011}
		Bauschke, H.H., Combettes, P.L. \textit{Convex Analysis and Monotone Operator Theory in Hilbert Spaces}, 2nd ed. Springer, Cham, (2017).
		
		\bibitem{Bred}
		Bredies, K., Chenchene, E., Lorenz, D.A. et al. Degenerate preconditioned proximal point algorithms. SIAM Journal on Optimization. 32(3), 2376--2401 (2022).	
		
		\bibitem{Bredies-Graph2024}
		Bredies, K., Chenchene, E., Naldi, E. Graph and distributed extensions of the Douglas--Rachford method. SIAM Journal on Optimization. 34(2), 1569--1594 (2024).
		
		\bibitem{minh}
		Dao, M.N., Tam, M.K., Truong, T.D. A general approach to distributed operator splitting. Journal of Mathematical Analysis and Applications. 562(2), 130692 (2026).
		
		\bibitem{DR}
		Douglas, J., Rachford, H. H. On the numerical solution of heat conduction problems in two and three space variables. Transactions of the American Mathematical Society. 82, 421--439 (1956).
		
		\bibitem{DRPR}
		Eckstein, J., Bertsekas, D. P. On the Douglas–Rachford splitting method and the proximal point algorithm for maximal monotone operators. Mathematical Programming. 55, 293--318 (1992).
		
		\bibitem{AAAI}
		Farina, G., Kroer, C., Sandholm, T. Online convex optimization for sequential decision processes and extensive-form games. Proceedings of the AAAI Conference on Artificial Intelligence. 33(1), 1917--1925 (2019).
		
		\bibitem{FBWR}
		Fu, Y., Zhao, J., Dong, Q.L. et al. Frugal forward-backward splitting methods with reflection terms. Journal of Nonlinear and Convex Analysis. accepted.
		
		\bibitem{GAP}
		Hamedani, E.Y., Aybat, N.S. A primal-dual algorithm with line search for general convex-concave saddle point problems.
		SIAM Journal on Optimization. 31(2), 1299--1329 (2021).
		
		\bibitem{Jiang2025}
		Jiang, R., Mokhtari, A. Generalized optimistic methods for convex-concave saddle point problems. SIAM Journal on Optimization. 35(3), 2066--2097 (2025).
		
		\bibitem{DYR}
		Ke, S., Dong, Q.L., Petru\c{s}el, A. et al. Generalized reflected Davis--Yin splitting algorithms for the monotone inclusions problems. Fixed Point Theory. 27, 327--348 (2026).
		
		\bibitem{n=m=1}
		Lions, P.L., Mercier, B. Splitting algorithms for the sum of two nonlinear operators. SIAM Journal on Numerical Analysis. 16, 964--979 (1979).
		
		\bibitem{MKT}
		Malitsky, Y., Tam, M.K. Resolvent splitting for sums of monotone operators with minimal lifting. Mathematical Programming. 201, 231--262 (2023).
		
		
		\bibitem{Pierra}
		Pierra, G. Decomposition through formalization in a product space. Mathematical Programming. 28, 96--115 (1984).
		
		\bibitem{Rockafellar1970}
		Rockafellar, R.T. On the maximal monotonicity of subdifferential mappings.
		Pacific Journal of Mathematics. 33(1), 209--216 (1970).
		
		\bibitem{saddlemin}
		Rockafellar, R.T. Monotone operators associated with saddle-functions and minimax problems. American Mathematical Society. 18(1), 241--250 (1970).
		
		
		
		\bibitem{proxi}
		Rockafellar, R.T. Monotone operators and the proximal point algorithm. SIAM Journal on Control and Optimization. 14, 877--898 (1976).
		
		
		\bibitem{Ryu}
		Ryu, E.K. Uniqueness of DRS as the 2 operator resolvent-splitting and impossibility of 3 operator resolvent-splitting. Mathematical Programming. 182, 233--273 (2020).
		
		\bibitem{NO}
		Simonetto, A., Keviczky, T., Johansson, M. A regularized saddle-point algorithm for networked optimization with resource allocation constraints. 2012 IEEE 51st Conference on Decision and Control(CDC). 7476--7481 (2012).
		
		
		
		
		
		
		
		\bibitem{Tam}
		Tam, M.K. Frugal and decentralised resolvent splittings defined by nonexpansive operators. Optimization Letters. 18, 1815--1837 (2024).
		
		\bibitem{Tseng}
		Tseng, P. A modified forward-backward splitting method for maximal monotone mappings. SIAM Journal on Control and Optimization. 38, 431--446 (2000).
		
		\bibitem{Zhao2025}
		Zhao, L., Zhu, D., Zhang, S. An augmented Lagrangian approach to conically constrained nonmonotone variational inequality problems.
		Mathematics of Operations Research. 50(3), 1868--1900 (2025).
		
		\bibitem{primal-2026}
		Zhou, L., Ma, F. A symmetric primal-dual method with double extrapolation for composite convex optimization involving three functions.
		Computational Optimization and Applications. 94, 111--140 (2026).
		
		
		
		
		
		%
		
		%
		
		
		%
		
		
		
		
		
		
		
		
		
		
		
		
	\end{thebibliography}

	\section*{Appendix: Realisations of the framework}
	
	We provide several realisations of Algorithm \ref{AlgG} for solving a special case of the problem \eqref{ABC}, all of which are taken from \cite{FBWR}.
	
	We consider the problem which is to find $x\in \mathcal{H}$ such that
	\begin{equation*}
		\label{ABC1}
		0  \in \left ( \sum_{i=1}^{n}A_i+\sum_{i=1}^{n-1}B_i +\sum_{i=1}^{n-2}C_i  \right )(x),
	\end{equation*}
	where $A_1,\dots,A_n : \mathcal{H} \rightrightarrows \mathcal{H}$ are  maximally monotone,  $B_i: \mathcal{H}\to \mathcal{H} $ are $\sigma_i$-cocoercive,  $i=1,\dots,n-1$ and $C_i: \mathcal{H}\to \mathcal{H} $ are monotone and $L_i$-Lipschitz continuous, $i=1,\dots,n-2$.
	Let  $\Sigma=\diag({\sigma_1},\dots,\sigma_{n-1})$ and $L=\diag({L_1},\dots,L_{n-2})$ and  set $\overline{L}=\max_{1\le i\le n-2}L_i$ and $\underline{\sigma}=\min_{1\le i\le n-1}\sigma_i$.
	\newcounter{specialsec}
	\renewcommand{\thesubsection}{A.\arabic{specialsec}}
	\stepcounter{specialsec}
	\subsection{The first realisation: Distributed forward-backward method with reflection terms}
	
	We select  matrices as follows
	\begin{equation*}
		\begin{aligned}
			K=
			&\begin{pmatrix}
				\frac{2}{d}&-\frac{1}{d}&0&\cdots&0&-\frac{1}{d}\\
				-\frac{1}{d}&\frac{2}{d}&-\frac{1}{d}&\cdots&0&0\\
				0&-\frac{1}{d}&\frac{2}{d}&\cdots&0&0\\
				\vdots&\vdots&\vdots&\ddots&\vdots&\vdots\\
				0&0&0&\cdots&\frac{2}{d}&-\frac{1}{d}\\
				-\frac{1}{d}&0&0&\cdots&-\frac{1}{d}&\frac{2}{d}\\
			\end{pmatrix}_{n\times n},
			M=\lambda
			\begin{pmatrix}
				1&0&\cdots&0&0\\
				-1&1&\cdots&0&0\\
				0&-1&\cdots&0&0\\
				\vdots&\vdots&\ddots&\vdots&\vdots\\
				0&0&\cdots&-1&1\\
				0&0&\cdots&0&-1\\
			\end{pmatrix}_{n\times (n-1)},\\
			P=
			&\begin{pmatrix}
				0&0&\cdots&0\\
				1&0&\cdots&0\\
				0&1&\cdots&0\\
				\vdots&\vdots&\ddots&\vdots\\
				0&0&\cdots&1\\
				0&0&\cdots&0\\
			\end{pmatrix}_{n\times (n-2)},
			Q=
			\begin{pmatrix}
				0&0&\cdots&0\\
				0&0&\cdots&0\\
				1&0&\cdots&0\\
				0&1&\cdots&0\\
				\vdots&\vdots&\ddots&\vdots\\
				0&0&\cdots&1\\
			\end{pmatrix}_{n\times (n-2)},
			R=
			\begin{pmatrix}
				1&0&\cdots&0&0&0\\
				0&1&\cdots&0&0&0\\
				\vdots&\vdots&\ddots&\vdots&\vdots&\vdots\\
				0&0&\cdots&1&0&0\\
			\end{pmatrix}_{(n-2)\times n},\\
			H=
			&\begin{pmatrix}
				0&0&\cdots&0\\
				1&0&\cdots&0\\
				0&1&\cdots&0\\
				\vdots&\vdots&\ddots&\vdots\\
				0&0&\cdots&1\\
			\end{pmatrix}_{n\times (n-1)},
			G=
			\begin{pmatrix}
				1&0&\cdots&0&0\\
				0&1&\cdots&0&0\\
				\vdots&\vdots&\ddots&\vdots&\vdots\\
				0&0&\cdots&1&0\\
			\end{pmatrix}_{(n-1)\times n},
		\end{aligned}
	\end{equation*}
	where  $d,\lambda>0$. It is easy to verify that
	$e_n^\top K e_n=0$ and  Assumption \ref{asm11}(a)\&(b) are satisfied.  To ensure the convergence, we must have
	\begin{equation}\label{18}
		\begin{aligned}
			0\preceq& K-MM^{\top}-\frac{1}{2}(H-G^{\top}){\Sigma^{-1}}(H^{\top}-G)\\
			&-
			(P-Q)L(P^\top-Q^\top)-(P-R^{\top})L(P^\top-R).
		\end{aligned}
	\end{equation}
	The right-hand side of (\ref{18}) can be expressed as:
	{\footnotesize
		\begin{equation*}\label{tt}
			\begin{aligned}
				&\begin{pmatrix}
					\frac{2}{d}&-\frac{1}{d}&0&\cdots&0&-\frac{1}{d}\\
					-\frac{1}{d}&\frac{2}{d}&-\frac{1}{d}&\cdots&0&0\\
					0&-\frac{1}{d}&\frac{2}{d}&\cdots&0&0\\
					\vdots&\vdots&\vdots&\ddots&\vdots&\vdots\\
					0&0&0&\cdots&\frac{2}{d}&-\frac{1}{d}\\
					-\frac{1}{d}&0&0&\cdots&-\frac{1}{d}&\frac{2}{d}\\
				\end{pmatrix}
				-\lambda^2
				\begin{pmatrix}
					1&-1&0&\cdots&0&0\\
					-1&2&-1&\cdots&0&0\\
					0&-1&2&\cdots&0&0\\
					\vdots&\vdots&\vdots&\ddots&\vdots&\vdots\\
					0&0&0&\cdots&2&-1\\
					0&0&0&\cdots&-1&1
				\end{pmatrix}\\
				&-\frac{1}{2}
				\begin{pmatrix}
					\frac{1}{\sigma_1} & -\frac{1}{\sigma_1} & 0 & \dots & 0 & 0 \\
					-\frac{1}{\sigma_1} & \frac{1}{\sigma_1}+\frac{1}{\sigma_2} & -\frac{1}{\sigma_2} & \dots & 0 & 0 \\
					0 & -\frac{1}{\sigma_2} & \frac{1}{\sigma_2}+\frac{1}{\sigma_3} & \dots & 0 & 0 \\
					\vdots & \vdots & \vdots & \ddots & \vdots & \vdots \\
					0 & 0 & 0 & \dots & \frac{1}{\sigma_{n-2}}+\frac{1}{\sigma_{n-1}} & -\frac{1}{\sigma_{n-1}} \\
					0 & 0 & 0 & \dots & -\frac{1}{\sigma_{n-1}} & \frac{1}{\sigma_{n-1}}
				\end{pmatrix}\\
				&-
				\begin{pmatrix}
					L_1&-L_1&0&\cdots&0&0&0\\
					-L_1&2L_1+L_2&-L_2-L_1&\cdots&0&0&0\\
					0&-L_1-L_2&L_1+2L_2+L_3&\cdots&0&0&0\\
					\vdots&\vdots&\vdots&\ddots&\vdots&\vdots&\vdots\\
					0&0&0&\cdots&L_{n-4}+2L_{n-3}+L_{n-2}&-L_{n-3}-L_{n-2}&0\\
					0&0&0&\cdots&-L_{n-3}-L_{n-2}&L_{n-3}+2L_{n-2}&-L_{n-2}\\
					0&0&0&\cdots&0&-L_{n-2}&L_{n-2}\\
				\end{pmatrix}.
			\end{aligned}
		\end{equation*}
	}
	To ensure that the above matrix is negative semidefinite, the following conditions need to be satisfied:
	\begin{equation*}
		\left\{
		\begin{aligned}
			\frac{2}{d}-\lambda^2-\frac{1}{2\sigma_1}-L_1&\geq0\\
			\frac{2}{d}-2\lambda^2-\frac{1}{2\sigma_1}-\frac{1}{2\sigma_{2}}-2L_{1}-L_{2}&\geq0\\ \frac{2}{d}-2\lambda^2-\frac{1}{2\sigma_i}-\frac{1}{2\sigma_{i+1}}-L_{i-1}-2L_{i}-L_{i+1}&\geq0,~i=2,...,n-3,\\
			\frac{2}{d}-2\lambda^2-\frac{1}{2\sigma_{n-2}}-\frac{1}{2\sigma_{n-1}}-L_{n-3}-2L_{n-2}&\geq0\\
			\frac{2}{d}-\lambda^2-\frac{1}{2\sigma_{n-1}}-L_{n-2}&\geq0,\\
		\end{aligned}
		\right.
	\end{equation*}
	which can be reformulated as
	\begin{equation}\label{ww}
		\left\{
		\begin{aligned}
			2-\frac{d}{2\sigma_1}-dL_1&\geq\lambda^2d\\
			1-\frac{d}{4\sigma_1}-\frac{d}{4\sigma_{2}}-\frac{dL_{2}}{2}-dL_{1}&\geq\lambda^2d\\ 1-\frac{d}{4\sigma_i}-\frac{d}{4\sigma_{i+1}}-\frac{dL_{i-1}}{2}-dL_i-\frac{dL_{i+1}}{2}&\geq\lambda^2d,~~~~i=2,...,n-3,\\
			1-\frac{d}{4\sigma_{n-2}}-\frac{d}{4\sigma_{n-1}}-\frac{dL_{n-3}}{2}-dL_{n-2}&\geq\lambda^2d\\
			2-\frac{d}{2\sigma_{n-1}}-dL_{n-2}&\geq\lambda^2d.\\
		\end{aligned}
		\right.
	\end{equation}
	Since $d>0,$ it follows from \eqref{ww} that  $d<\overline d$ where $\overline d=\text{min}\{\frac{2}{\frac{1}{2\sigma_{1}}+L_1},  \frac{1}{\frac{1}{4\sigma_1}+\frac{1}{4\sigma_2}+\frac{L_2}{2}+L_1},\\
	\min_{2\le i\le n-3}\frac{1}{\frac{1}{4\sigma_i}
		+\frac{1}{4\sigma_{i+1}}+\frac{L_{i-1}}{2}+L_i+\frac{L_{i+1}}{2}},$ $\frac{1}{\frac{1}{4\sigma_{n-2}}+
		\frac{1}{4\sigma_{n-1}}+\frac{L_{n-3}}{2}+L_{n-2}},\frac{2}{\frac{1}{2\sigma_{n-1}}+L_{n-2}}\}.$
	It is easy to conclude that $\overline d$ is larger when $C_i$ with larger $L_i$ are placed in the first or last evaluations or when $B_i$ with larger $\sigma_i$ and $C_i$ with larger $L_i$  are grouped together.
	
	Let $\widetilde{z}_i^k:=\lambda d z_i^k$, $i=1,\ldots,n-1$, $\gamma_k\equiv\gamma$ for $k\ge0$ and $\widetilde{\gamma}=\lambda^2d\gamma$, then Algorithm \ref{AlgG} becomes
	\begin{equation*}
		\label{AlgA1}
		\left\{
		\begin{aligned}
			&x_1^k=J_{dA_1}(\widetilde{z}_1^k)\\
			&x_2^k=J_{dA_2}(\widetilde{z}_2^k+x_1^k-\widetilde{z}_1^k-dB_1(x_1^k)-dC_1(x_1^k))\\
			&x_i^{k} = J_{dA_i}(\widetilde{z}_i^k+x_{i-1}^k-\widetilde{z}_{i-1}^k-dB_{i-1}(x_{i-1}^k)-dC_{i-1}
			(x_{i-1}^k)\\
			&\qquad-d(C_{i-2}(x_{i-1}^k)-C_{i-2}(x_{i-2}^k))),\quad i=3,...,n-1,\\
			&x_n^{k} = J_{dA_n}(x_{1}^k+x_{n-1}^k-\widetilde{z}_{i-1}^k-dB_{n-1}(x_{n-1}^k)-d(C_{n-2}(x_{n-1}^k)-C_{n-2}(x_{n-2}^k)))\\
			&\widetilde{z}_i^{k+1}=\widetilde{z}_i^k+\widetilde{\gamma}(x_{i+1}^k-x_i^k),
			\quad i =1,\ldots,n-1,
		\end{aligned}
		\right.
	\end{equation*}
	which is exactly  \cite[Algorithm 3.1]{FBWR}.
	Due to $\gamma\in(0,1)$ and \eqref{ww}, we have $\widetilde \gamma\in(0,1-\Gamma)$ where  $\Gamma=d\max\{-\frac1d+ \frac1{2\sigma_1}+L_1, \frac1{4\sigma_1}+\frac1{4\sigma_2}+\frac{L_2}2+L_1, \max_{2\le i\le n-3}\{\frac1{4\sigma_i}+\frac1{4\sigma_{i+1}}+\frac{L_{i-1}}2+L_i+\frac{L_{i+1}}2\}, \frac1{4\sigma_{n-2}}+\frac1{4\sigma_{n-1}}+\frac{L_{n-3}}2+L_{n-2},-\frac1d+
	\frac1{2\sigma_{n-1}}+L_{n-2}\}$.
	Since $\overline d\ge \frac{2\underline{\sigma}}{1+4\underline{\sigma }\overline{L}}$ and $\Gamma\le\frac{d(1+4\underline{\sigma}\overline L)}{2\underline{\sigma}}$,  the ranges of $d$ and $\widetilde\gamma$ include
	those given  in \cite{FBWR}.  Therefore,  we can always  enlarge the range of $d$ or $\widetilde\gamma$ given in \cite{FBWR} by changing the order of operators when the values of $L_i$, $i=1,\ldots,n-2$ (or $\sigma_i$, $i=1,\ldots,n-1$) are not all equal.
	
	\stepcounter{specialsec}
	\subsection{The second realisation: Generalized parallel Davis--Yin method with reflected terms}\label{a2}
	We select  matrices as follows
	\begin{equation*}
		\begin{aligned}
			K=&
			\begin{pmatrix}
				\frac{2(n-1)}{d}&-\frac{2}{d}&-\frac{2}{d}&\cdots&-\frac{2}{d}&-\frac{2}{d}\\
				-\frac{2}{d}&\frac{2}{d}&0&\cdots&0&0\\
				-\frac{2}{d}&0&\frac{2}{d}&\cdots&0&0\\
				\vdots&\vdots&\vdots&\ddots&\vdots&\vdots\\
				-\frac{2}{d}&0&0&\cdots&\frac{2}{d}&0\\
				-\frac{2}{d}&0&0&\cdots&0&\frac{2}{d}\\
			\end{pmatrix}_{n\times n},
			M=\lambda
			\begin{pmatrix}
				1&1&\cdots&1&1\\
				-1&0&\cdots&0&0\\
				0&-1&\cdots&0&0\\
				\vdots&\vdots&\ddots&\vdots&\vdots\\
				0&0&\cdots&-1&0\\
				0&0&\cdots&0&-1\\
			\end{pmatrix}_{n\times (n-1)},\\
			P=&
			\begin{pmatrix}
				0&0&\cdots&0\\
				1&0&\cdots&0\\
				0&1&\cdots&0\\
				\vdots&\vdots&\ddots&\vdots\\
				0&0&\cdots&1\\
				0&0&\cdots&0\\
			\end{pmatrix}_{n\times (n-2)},
			Q=
			\begin{pmatrix}
				0&0&\cdots&0\\
				0&0&\cdots&0\\
				1&0&\cdots&0\\
				0&1&\cdots&0\\
				\vdots&\vdots&\ddots&\vdots\\
				0&0&\cdots&1\\
			\end{pmatrix}_{n\times (n-2)},
			R=
			\begin{pmatrix}
				1&0&\cdots&0&0&0\\
				1&0&\cdots&0&0&0\\
				\vdots&\vdots&\ddots&\vdots&\vdots&\vdots\\
				1&0&\cdots&0&0&0\\
			\end{pmatrix}_{(n-2)\times n}\\
			H=&
			\begin{pmatrix}
				0&0&\cdots&0\\
				1&0&\cdots&0\\
				0&1&\cdots&0\\
				\vdots&\vdots&\ddots&\vdots\\
				0&0&\cdots&1\\
			\end{pmatrix}_{n\times (n-1)},
			G=
			\begin{pmatrix}
				1&0&\cdots&0&0\\
				1&0&\cdots&0&0\\
				\vdots&\vdots&\ddots&\vdots&\vdots\\
				1&0&\cdots&0&0\\
			\end{pmatrix}_{(n-1)\times n},
		\end{aligned}
	\end{equation*}
	where $d,\lambda>0.$ It is easy to verify that
	$e_n^\top K e_n=0$ and  Assumption \ref{asm11}(a)\&(b) are satisfied. To guarantee the convergence, we require
	\begin{equation}\label{banfuding}
		\begin{aligned}
			0\preceq& K-MM^{\top}-\frac{1}{2}(H-G^{\top}){\Sigma^{-1}}(H^{\top}-G)\\
			&-
			(P-Q)L(P^\top-Q^\top)-(P-R^{\top})L(P^\top-R).
		\end{aligned}
	\end{equation}
	The right-hand side of (\ref{banfuding}) can be expressed as:
	\begin{equation*}\label{juzhen}
		\begin{aligned}
			&\begin{pmatrix}
				\frac{2(n-1)}{d}&-\frac{2}{d}&-\frac{2}{d}&\cdots&-\frac{2}{d}&-\frac{2}{d}\\
				-\frac{2}{d}&\frac{2}{d}&0&\cdots&0&0\\
				-\frac{2}{d}&0&\frac{2}{d}&\cdots&0&0\\
				\vdots&\vdots&\vdots&\ddots&\vdots&\vdots\\
				-\frac{2}{d}&0&0&\cdots&\frac{2}{d}&0\\
				-\frac{2}{d}&0&0&\cdots&0&\frac{2}{d}\\
			\end{pmatrix}
			-\lambda^2
			\begin{pmatrix}
				n-1&-1&-1&\cdots&-1&-1\\
				-1&1&0&\cdots&0&0\\
				-1&0&1&\cdots&0&0\\
				\vdots&\vdots&\vdots&\ddots&\vdots&\vdots\\
				-1&0&0&\cdots&1&0\\
				-1&0&0&\cdots&0&1\\
			\end{pmatrix}\\
			&-\frac{1}{2}				\begin{pmatrix}
				\sum_{i=1}^{n-1}\frac{1}{\sigma_i}&-\-\frac{1}{\sigma_1}&-\frac{1}{\sigma_2}&\cdots&-\frac{1}{\sigma_{n-2}}&-\frac{1}{\sigma_{n-1}}\\
				-\frac{1}{\sigma_1}&\frac{1}{\sigma_1}&0&\cdots&0&0\\
				-\frac{1}{\sigma_2}&0&\frac{1}{\sigma_2}&\cdots&0&0\\
				\vdots&\vdots&\vdots&\ddots&\vdots&\vdots\\
				-\frac{1}{\sigma_{n-2}}&0&0&\cdots&\frac{1}{\sigma_{n-2}}&0\\
				-\frac{1}{\sigma_{n-1}}&0&0&\cdots&0&\frac{1}{\sigma_{n-1}}\\
			\end{pmatrix}\\
			&-
			\begin{pmatrix}
				\sum_{i=1}^{n-2}L_i&-L_1&-L_2&\cdots&-L_{n-3}&-L_{n-2}&0\\
				-L_1&2L_1&-L_1&\cdots&0&0&0\\
				-L_2&-L_1&L_1+2L_2&\cdots&0&0&0\\
				\vdots&\vdots&\vdots&\vdots&\ddots&\vdots&\vdots\\
				-L_{n-3}&0&0&\cdots&L_{n-4}+2L_{n-3}&-L_{n-3}&0\\
				-L_{n-2}&0&0&\cdots&-L_{n-3}&L_{n-3}+2L_{n-2}&-L_{n-2}\\
				0&0&0&\cdots&0&-L_{n-2}&L_{n-2}\\
			\end{pmatrix}.
		\end{aligned}
	\end{equation*}
	To ensure that the above matrix is negative semidefinite, the following conditions need to be satisfied:
	\begin{equation*}\label{3.2}
		\left\{
		\begin{aligned}
			\frac{2(n-1)}{d}-\lambda^2(n-1)-\frac12\sum_{i=1}^{n-1}\frac{1}{\sigma_i}-\sum_{i=1}^{n-2}L_i&\geq0\\
			\frac{2}{d}-\lambda^2-\frac{1}{2\sigma_1}-2L_1&\geq0\\
			\frac{2}{d}-\lambda^2-\frac{1}{2\sigma_i}-L_{i-1}-2L_i&\geq0,\,\,i=2,\ldots,n-2\\
			\frac{2}{d}-\lambda^2-\frac{1}{2\sigma_{n-1}}-L_{n-2}&\geq0,
		\end{aligned}
		\right.
	\end{equation*}
	which equals
	\begin{equation}\label{3.21}
		\left\{
		\begin{aligned}
			2(n-1)-d\sum_{i=1}^{n-1}\frac{1}{2\sigma_i}-d\sum_{i=1}^{n-2}L_i&\geq(n-1)\lambda^2d\\
			2-\frac{d}{2\sigma_1}-2dL_1&\geq\lambda^2d\\
			2-\frac{d}{2\sigma_i}-dL_{i-1}-2dL_i&\geq\lambda^2d,\,\,i=2,\ldots,n-2\\
			2-\frac{d}{2\sigma_{n-1}}-dL_{n-2}&\geq\lambda^2d.\\
		\end{aligned}
		\right.
	\end{equation}
	Since $d>0,$ it follows from \eqref{3.21} that  $d<\overline d$ where $\bar d=\text{min}\{\frac{2(n-1)}{\sum_{i=1}^{n-1}\frac{1}{2\sigma_i}+\sum_{i=1}^{n-2}L_i},
	\frac{2}{\frac{1}{2\sigma_1}+2L_1},\\
	\min_{2\le i\le n-2}\frac{2}{\frac{1}{2\sigma_i}+L_{i-1}+2L_i},\frac{2}
	{\frac{1}{2\sigma_{n-1}}+L_{n-2}}\}.$
	It can be observed that $\overline d$ is larger when $C_i$ with the second-largest and the largest   $L_i$ are placed in the second and last evaluations or when $B_i$ with larger $\sigma_i$ and $C_i$ with larger $L_i$  are grouped together.
	
	Let $\widetilde{z}_i^k:=\lambda d z_i^k$, $i=1,\ldots,n-1$, $\gamma_k\equiv\gamma$ for $k\ge0$ and $\widetilde{\gamma}=\lambda^2d\gamma$, then Algorithm \ref{AlgG} becomes
	\begin{equation*}
		\label{AlgA2}
		\left\{
		\begin{aligned}
			&x_1^k=J_{\frac{d}{n-1}A_1}(\textstyle\frac{1}{n-1}\sum_{i=1}^{n-1}\widetilde{z}_i^k)\\
			&x_2^k=J_{dA_2}(2x_1^k-\widetilde{z}_1^k-dB_1(x_1^k)-dC_1(x_1^k))\\
			&x_i^{k} = J_{dA_i}(2x_1^k-\widetilde{z}_{i-1}^k-dB_{i-1}(x_{1}^k)-dC_{i-1}(x_{1}^k)-d(C_{i-2}(x_{i-1}^k)-C_{i-2}(x_{1}^k)))\\				&~~~~~~~~~~~~~~~~~~~~~~~~~~~~~~~~~~~~~~~~~~~~~~~~~~~~~~~~~~~~~~~~~~~~~~~~~~~~~~~~~~~~~~~~~i=3,...,n-1,\\
			&x_n^{k} = J_{dA_n}(2x_{1}^k-\widetilde{z}_{n-1}^k-dB_{n-1}(x_{1}^k)-d(C_{n-2}(x_{n-1}^k)-C_{n-2}(x_{1}^k)))\\
			&\widetilde{z}_i^{k+1}=\widetilde{z}_i^k+\widetilde{\gamma}(x_{i+1}^k-x_1^k),
			\quad i =1,\ldots,n-1,
		\end{aligned}
		\right.
	\end{equation*}
	which is exactly \cite[Algorithm 3.2]{FBWR}.
	Due to $\gamma\in(0,1)$ and \eqref{3.21}, we have $\widetilde \gamma\in(0,2-\Gamma)$ where  $\Gamma=d\max\{\frac 1{2(n-1)}\sum_{i=1}^{n-1}\frac{1}{\sigma_i}+\frac 1{n-1}\sum_{i=1}^{n-2}L_i,\frac{1}{2\sigma_1}+2L_1,\max_{2\le i\le n-2}\{\frac{1}{2\sigma_i}+L_{i-1}+2L_i\},
	\frac{1}{2\sigma_{n-1}}+L_{n-2}\}$.
	Since $\overline d\ge \frac{4\underline{\sigma}}{1+6\underline{\sigma }\overline{L}}$ and $\Gamma\le\frac{d(1+6\underline{\sigma}\overline L)}{2\underline{\sigma}}$, the ranges of $d$ and $\widetilde\gamma$  are larger than those  given in \cite{FBWR}.
	
	\stepcounter{specialsec}
	\subsection{The third realisation: Generalized sequential Davis--Yin method with reflected terms}	
	We select  matrices as follows
	\newpage
	\begin{equation*}
		\begin{aligned}
			K=&\begin{pmatrix}
				\frac{2}{d}&-\frac{2}{d}&0&\cdots&0&0\\
				-\frac{2}{d}&\frac{4}{d}&-\frac{2}{d}&\cdots&0&0\\
				0&-\frac{2}{d}&\frac{4}{d}&\cdots&0&0\\
				\vdots&\vdots&\vdots&\ddots&\vdots&\vdots\\
				0&0&0&\cdots&\frac{4}{d}&-\frac{2}{d}\\
				0&0&0&\cdots&-\frac{2}{d}&\frac{2}{d}\\
			\end{pmatrix}_{n\times n},
			M=\lambda
			\begin{pmatrix}
				1&0&\cdots&0&0\\
				-1&1&\cdots&0&0\\
				0&-1&\cdots&0&0\\
				\vdots&\vdots&\ddots&\vdots&\vdots\\
				0&0&\cdots&-1&1\\
				0&0&\cdots&0&-1\\
			\end{pmatrix}_{n\times (n-1)},\\
			P=&
			\begin{pmatrix}
				0&0&\cdots&0\\
				1&0&\cdots&0\\
				0&1&\cdots&0\\
				\vdots&\vdots&\ddots&\vdots\\
				0&0&\cdots&1\\
				0&0&\cdots&0\\
			\end{pmatrix}_{n\times (n-2)},
			Q=
			\begin{pmatrix}
				0&0&\cdots&0\\
				0&0&\cdots&0\\
				1&0&\cdots&0\\
				0&1&\cdots&0\\
				\vdots&\vdots&\ddots&\vdots\\
				0&0&\cdots&1\\
			\end{pmatrix}_{n\times (n-2)},
			R=
			\begin{pmatrix}
				1&0&\cdots&0&0&0\\
				0&1&\cdots&0&0&0\\
				\vdots&\vdots&\ddots&\vdots&\vdots&\vdots\\
				0&0&\cdots&1&0&0\\
			\end{pmatrix}_{(n-2)\times n}\\
			H=&
			\begin{pmatrix}
				0&0&\cdots&0\\
				1&0&\cdots&0\\
				0&1&\cdots&0\\
				\vdots&\vdots&\ddots&\vdots\\
				0&0&\cdots&1\\
			\end{pmatrix}_{n\times (n-1)},
			G=
			\begin{pmatrix}
				1&0&\cdots&0&0\\
				0&1&\cdots&0&0\\
				\vdots&\vdots&\ddots&\vdots&\vdots\\
				0&0&\cdots&1&0\\
			\end{pmatrix}_{(n-1)\times n},
		\end{aligned}
	\end{equation*}
	where $d,\lambda>0.$ It is easy to verify that
	$e_n^\top K e_n=0$ and  Assumption \ref{asm11}(a)\&(b) are satisfied. To ensure the convergence, we must have
	\begin{equation}\label{26}
		\begin{aligned}
			0\preceq& K-MM^{\top}-\frac{1}{2}(H-G^{\top}){\Sigma^{-1}}(H^{\top}-G)\\
			&-
			(P-Q)L(P^\top-Q^\top)-(P-R^{\top})L(P^\top-R).
		\end{aligned}
	\end{equation}
	The right-hand side of (\ref{26}) can be expressed as:
	{\footnotesize
		\begin{equation*}\label{27}
			\begin{aligned}
				&\begin{pmatrix}
					\frac{2}{d}&-\frac{2}{d}&0&\cdots&0&0\\
					-\frac{2}{d}&\frac{4}{d}&-\frac{2}{d}&\cdots&0&0\\
					0&-\frac{2}{d}&\frac{4}{d}&\cdots&0&0\\
					\vdots&\vdots&\vdots&\ddots&\vdots&\vdots\\
					0&0&0&\cdots&\frac{4}{d}&-\frac{2}{d}\\
					0&0&0&\cdots&-\frac{2}{d}&\frac{2}{d}\\
				\end{pmatrix}
				-\lambda^2
				\begin{pmatrix}
					1&-1&0&\cdots&0&0\\
					-1&2&-1&\cdots&0&0\\
					0&-1&2&\cdots&0&0\\
					\vdots&\vdots&\vdots&\ddots&\vdots&\vdots\\
					0&0&0&\cdots&2&-1\\
					0&0&0&\cdots&-1&1
				\end{pmatrix}\\
				&-\frac{1}{2}
				\begin{pmatrix}
					\frac{1}{\sigma_1} & -\frac{1}{\sigma_1} & 0 & \dots & 0 & 0 \\
					-\frac{1}{\sigma_1} & \frac{1}{\sigma_1}+\frac{1}{\sigma_2} & -\frac{1}{\sigma_2} & \dots & 0 & 0 \\
					0 & -\frac{1}{\sigma_2} & \frac{1}{\sigma_2}+\frac{1}{\sigma_3} & \dots & 0 & 0 \\
					\vdots & \vdots & \vdots & \ddots & \vdots & \vdots \\
					0 & 0 & 0 & \dots & \frac{1}{\sigma_{n-2}}+\frac{1}{\sigma_{n-1}} & -\frac{1}{\sigma_{n-1}} \\
					0 & 0 & 0 & \dots & -\frac{1}{\sigma_{n-1}} & \frac{1}{\sigma_{n-1}}
				\end{pmatrix}\\
				&-
				\begin{pmatrix}
					L_1&-L_1&0&\cdots&0&0&0\\
					-L_1&2L_1+L_2&-L_1-L_2&\cdots&0&0&0\\
					0&-L_1-L_2&L_1+2L_2+L_3&\cdots&0&0&0\\
					\vdots&\vdots&\vdots&\ddots&\vdots&\vdots&\vdots\\
					0&0&0&\cdots&L_{n-4}+2L_{n-3}+L_{n-2}&-L_{n-3}-L_{n-2}&0\\
					0&0&0&\cdots&-L_{n-3}-L_{n-2}&L_{n-3}+2L_{n-2}&-L_{n-2}\\
					0&0&0&\cdots&0&-L_{n-2}&L_{n-2}\\
				\end{pmatrix}.
			\end{aligned}
		\end{equation*}
	}
	To ensure that the above matrix is negative semidefinite, the following conditions need to be satisfied:
	\begin{equation*}\label{28}
		\left\{
		\begin{aligned}
			\frac{2}{d}-\lambda^2-\frac{1}{2\sigma_1}-L_1&\geq0\\
			\frac{4}{d}-2\lambda^2-\frac{1}{2\sigma_1}-\frac{1}{2\sigma_{2}}-L_{2}-2L_{1}&\geq0\\
			\frac{4}{d}-2\lambda^2-\frac{1}{2\sigma_i}-\frac{1}{2\sigma_{i+1}}-L_{i-1}-2L_{i}-L_{i+1}&\geq0~i=2,...,n-3,\\
			\frac{4}{d}-2\lambda^2-\frac{1}{2\sigma_{n-2}}-\frac{1}{2\sigma_{n-1}}-L_{n-3}-2L_{n-2}&\geq0\\
			\frac{2}{d}-\lambda^2-\frac{1}{2\sigma_{n-1}}-L_{n-2}&\geq0,	\end{aligned}
		\right.
	\end{equation*}
	which can be rewritten as
	\begin{equation}\label{29}
		\left\{
		\begin{aligned}
			2-\frac{d}{2\sigma_1}-dL_1&\geq\lambda^2d\\
			2-\frac{d}{4\sigma_1}-\frac{d}{4\sigma_{2}}-\frac{dL_{2}}{2}-dL_{1}&\geq\lambda^2d\\
			2-\frac{d}{4\sigma_i}-\frac{d}{4\sigma_{i+1}}-\frac{dL_{i-1}}{2}-L_i-\frac{dL_{i+1}}{2}&\geq\lambda^2d,~~~~i=2,...,n-3,\\
			2-\frac{d}{4\sigma_{n-2}}-\frac{d}{4\sigma_{n-1}}-\frac{dL_{n-3}}{2}-dL_{n-2}&\geq\lambda^2d\\
			2-\frac{d}{2\sigma_{n-1}}-dL_{n-2}&\geq\lambda^2d.\\
		\end{aligned}
		\right.
	\end{equation}
	Since $d>0,$ it follows from \eqref{29} that  $d<\overline d$ where $\overline d=\min\{\frac{2}{\frac{1}{2\sigma_{1}}+L_1}, \frac{2}{\frac{1}{4\sigma_1}
		+\frac{1}{4\sigma_2}+\frac{L_2}{2}+L_1},\\
	\min_{2\le i\le n-3}\frac{2}{\frac{1}{4\sigma_i}+\frac{1}{4\sigma_{i+1}}+\frac{L_{i-1}}{2}
		+L_i+\frac{L_{i+1}}{2}},$$\frac{2}{\frac{1}{4\sigma_{n-2}}+\frac{1}{4\sigma_{n-1}}+
		\frac{L_{n-3}}{2}+L_{n-2}},\frac{2}{\frac{1}{2\sigma_{n-1}}+L_{n-2}}\}.$
	It can be observed that $\overline d$ is larger when $C_i$ with larger $L_i$ are placed in the first or last evaluations or when $B_i$ with larger $\sigma_i$ and $C_i$ with larger $L_i$  are grouped together.
	
	Let $\widetilde{z}_i^k:=\lambda d z_i^k$, $i=1,\ldots,n-1$, $\gamma_k\equiv\gamma$ for $k\ge0$ and $\widetilde{\gamma}=\lambda^2d\gamma$, then Algorithm \ref{AlgG} becomes
	\begin{equation*}\label{AlgA3}
		\left\{
		\begin{aligned}
			&x_1^k=J_{dA_1}(\widetilde{z}_1^k)\\
			&x_2^k=J_{\frac{d}{2}A_2}(x_1^k+\frac{\widetilde{z}_2^k-\widetilde{z}_1^k}{2}-\frac{d}{2}B_1(x_1^k)-\frac{d}{2}C_1(x_1^k))\\
			&x_i^{k} = J_{\frac{d}{2}A_i}(x_{i-1}^k+\frac{\widetilde{z}_i^k-\widetilde{z}_{i-1}^k}{2}-\frac{d}{2}B_{i-1}(x_{i-1}^k)-\frac{d}{2}C_{i-1}(x_{i-1}^k)-\frac{d}{2}(C_{i-2}(x_{i-1}^k)-C_{i-2}(x_{i-2}^k)))\\
			&~~~~~~~~~~~~~~~~~~~~~~~~~~~~~~~~~~~~~~~~~~~~~~~~~~~~~~~~~~~~~~~~~~~~~~~~~~~~~~~~~~~~~~~~~i=3,...,n-1,\\
			&x_n^{k} = J_{dA_n}(2x_{1}^k-\widetilde{z}_{i-1}^k-dB_{n-1}(x_{n-1}^k)-d(C_{n-2}(x_{n-1}^k)-C_{n-2}(x_{n-2}^k)))\\
			&\widetilde{z}_i^{k+1}=\widetilde{z}_i^k+\widetilde\gamma(x_{i+1}^k-x_i^k),
			\quad i =1,\ldots,n-1,
		\end{aligned}
		\right.
	\end{equation*}
	which is exactly \cite[Algorithm 3.3]{FBWR}.
	Due to $\gamma\in(0,1)$ and \eqref{29}, we have $\widetilde \gamma\in(0,2-\Gamma)$ where  $\Gamma=d\max\{ \frac1{2\sigma_1}+L_1, \frac1{4\sigma_1}+\frac1{4\sigma_2}+\frac{L_2}2+L_1, \max_{2\le i\le n-3}\{\frac1{4\sigma_i}+\frac1{4\sigma_{i+1}}+\frac{L_{i-1}}2+L_i+\frac{L_{i+1}}2\}, \frac1{4\sigma_{n-2}}+\frac1{4\sigma_{n-1}}+\frac{L_{n-3}}2+L_{n-2}, \frac1{2\sigma_{n-1}}+L_{n-2})\}$.
	Since $\overline d \ge \frac{4\underline{\sigma}}{1+4\underline{\sigma }\overline{L}}$ and $\Gamma\le \frac{d(1+4\underline{\sigma}\overline L)}{2\underline{\sigma}}$,  the ranges of $d$ and $\widetilde\gamma$ include those in \cite{FBWR}.  Therefore,  we can always  enlarge the range of $d$ or $\widetilde\gamma$ given in \cite{FBWR} by changing the order of operators when the values of $L_i$, $i=1,\ldots,n-2$ (or $\sigma_i$, $i=1,\ldots,n-1$) are not all equal.
	
	\newpage

\end{document}